\RequirePackage{amsmath}
\documentclass[runningheads]{llncs}
\usepackage[T1]{fontenc}
\usepackage{graphicx}
\usepackage{amsmath}
\usepackage[misc]{ifsym}
\usepackage{amsfonts}
\usepackage{algorithm}
\usepackage{algpseudocode}
\usepackage{pgffor}
\usepackage{nicefrac}
\usepackage{pgfplots}
\usepackage{subcaption}
\usepackage[caption=false]{subfig}

\usepackage{tikz}
\usetikzlibrary {automata,positioning,math}

\algtext*{EndWhile}
\algtext*{EndIf}
\algtext*{EndFor}
\algtext*{EndProcedure}

\newcommand{\MDP}[1]{\mathcal{M}_{#1}}
\newcommand{\Pol}[2]{\pi^\mathcal{#1}_{#2}}

\newcommand{\ReachSet}[2]{R^\mathcal{#1}_{#2}}
\newcommand{\NumRounds}{m_{r}}

\newcommand{\sd}[1]{S_{d,#1}}

\newcommand{\KL}[2]{\textrm{KL}\left(\Gamma^{#1}|| \Gamma^{#2}\right)}
\newcommand{\pathSeq}{\boldsymbol{\xi}}
\newcommand{\polSpace}[1]{\Pi(\mathcal{M}_{#1})}

\newcommand{\reachDisjunc}[2]{}

\newcommand{\isDec}[1]{#1 \ \textrm{is Deceptive}}
\newcommand{\isNotDec}[1]{#1 \ \textrm{is Not Deceptive}}
\newcommand{\DecPri}[1]{\Pr(\isDec{#1})}

\newcommand{\DecCond}[1]{\Pr(\pathSeq_#1|\isDec{#1})}
\newcommand{\NonDecCond}[1]{\Pr(\pathSeq_#1|\isNotDec{#1})}

\newcommand{\DecPriShort}[1]{p_{\mathcal{D},#1}}
\newcommand{\Belief}[1]{\theta_i(\boldsymbol{\xi}_{#1})}
\newcommand{\BeliefProxy}[2]{\hat{\theta}_{#1}(#2)}
\newcommand{\BeliefProxyTEXT}[2]{\smash{\hat{\theta}_{#1}(#2)}}

\newcommand{\AgProbSimple}{worst-case deceptive policy synthesis}
\newcommand{\AgProbSimpleTitle}{Worst-Case Deceptive Policy Synthesis}

\newcommand{\AgProbFull}{elimination-aware deceptive policy synthesis}
\newcommand{\AgProbFullTitle}{Elimination-Aware Deceptive Policy Synthesis}

\begin{document}
\title{A Decentralized Shotgun Approach for Team Deception}
%
%
\author{Caleb Probine\textsuperscript{(\Letter)} \and
Mustafa O. Karabag \and
Ufuk Topcu}
\authorrunning{C. Probine, M.O. Karabag, U. Topcu}
%
\institute{The University of Texas at Austin
\\ \email{\{cprobine,karabag,utopcu\}@utexas.edu}
}

\maketitle              
\begin{abstract}

Deception is helpful for agents masking their intentions from an observer. 
We consider a team of agents deceiving their supervisor. 
The supervisor defines nominal behavior for the agents via reference policies, but the agents share an alternate task that they can only achieve by deviating from these references.
As such, the agents use deceptive policies to complete the task while ensuring that their behaviors remain plausible to the supervisor.
We propose a setting with centralized deceptive policy synthesis and decentralized execution. 
We model each agent with a Markov decision process and constrain the agents' deceptive policies so that, with high probability, at least one agent achieves the task.
We then provide an algorithm to synthesize deceptive policies that ensure the deviations of all agents are small by minimizing the worst Kullback-Leibler divergence between any agent's deceptive and reference policies.
Thanks to decentralization, this algorithm scales linearly with the number of agents and also facilitates the efficient synthesis of reference policies.
We then explore a more general version of the deceptive policy synthesis problem.
In particular, we consider a supervisor who selects a subset of agents to eliminate based on the agents' behaviors. We give algorithms to synthesize deceptive policies so that, after the supervisor eliminates some agents, the remaining agents complete the task with high probability.
We demonstrate the developed methods in a package delivery example.

\end{abstract}
\section{Introduction}

In interactions with asymmetric information, agents can use deception to create an advantage against an opponent.
Examples of applications where deception is useful include human-robot interaction \cite{dragan2014} and intrusion or defense of cyber systems \cite{han2018,janczewski2007}. 
We consider a setting where a system manager
assigns policies to the system's components so that they complete a task in a distributed manner.
For example, one may assign decentralized policies to a team of aerial vehicles to complete a search task \cite{bahnemann2017}.
An adversary who gains control of the components, such as an external intruder, may change the components' policies to serve their own goals.
However, the system manager may supervise the agents. As such, the adversary must choose deceptive policies that deviate from assigned behavior in a plausible manner. Otherwise, the manager will detect the deviation.

We study the synthesis of deceptive policies for the components so that we may understand the weaknesses of these systems and improve their security.
To be consistent with \cite{karabag2021}, we label components as agents and the manager as their supervisor. Figure~\ref{fig:ProcessOutline} then shows the setup we consider.
The supervisor first assigns reference policies to the agents. The agents then collude to find deceptive policies so that the agents complete a shared alternate task. 
The agents must choose deceptive policies so that, after the supervisor observes the agents' behavior, they do not detect the agents' deviations and eliminate them.

\tikzmath{\a1 = 0.9; \xc = 1;
\x1 = \xc - \a1; \x2 = \xc ; \x3 = \xc + \a1;
\bRecW = 1.45;
\bRecH = 1;
\ya = 1.2;
\ys = 2.3;}
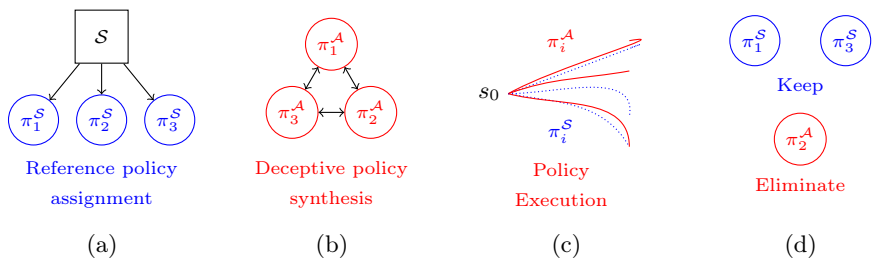
\begin{figure}[h]
\centering
\newcommand{\sfSep}{0.24}
\begin{subfigure}{\sfSep\textwidth}
\begin{tikzpicture}
    \pgfdeclarelayer{nodelayer}
    \pgfdeclarelayer{edgelayer}
    \pgfsetlayers{nodelayer,edgelayer}
    \begin{pgfonlayer}{nodelayer}
        \path (\x1,\ya) node[circle,draw,inner sep=2pt,blue](1) {\scriptsize $\pi^\mathcal{S}_1$}
            (\x2,\ya) node[circle,draw,inner sep=2pt,blue](2) {\scriptsize $\pi^\mathcal{S}_2$}
            (\x3,\ya) node[circle,draw,inner sep=2pt,blue](3) {\scriptsize $\pi^\mathcal{S}_3$}
              (\x2,\ys) node[rectangle,draw,minimum size = 0.7cm](s) {\scriptsize $\mathcal{S}$};
        \node[blue, align=center] at (\x2,\ya-0.85) {\scriptsize Reference policy \\ \scriptsize  assignment};
    \end{pgfonlayer}
	\begin{pgfonlayer}{edgelayer}
        \draw[rounded corners=0.5cm,draw opacity=0] (\xc - \bRecW,0) rectangle (\xc + \bRecW,\bRecH) {};
	      \draw[->] (s) -- (1);  
	      \draw[->] (s) -- (2);  
	      \draw[->] (s) -- (3);  
 \end{pgfonlayer}
\end{tikzpicture}
\subcaption{}
\end{subfigure}
\begin{subfigure}{\sfSep\textwidth}
\tikzmath{\n = 3; \r =0.6;\ya = 1.2;}
\begin{tikzpicture}
    \foreach \s in {1,...,\n }
    {
        \tikzmath{\x = \xc +  \r*sin(360/\n * (\s - 1)); 
                    \y = 1.6 + \r*cos(360/\n * (\s - 1));}
        \node[draw,circle,inner sep=2pt,red](\s) at (\x,\y) {\scriptsize $\pi^\mathcal{A}_\s$};
    }
    \tikzmath{\p = \n-1;}
    \foreach \s in {0,...,\p }
    {
        \tikzmath{\t = int(mod(\s+1,3) +1) ; \q = int(\s + 1);}
        \draw[<->] (\q) -- (\t);
    }
    \draw[rounded corners=0.5cm,draw opacity=0] (\xc - \bRecW,0) rectangle (\xc + \bRecW,\bRecH) {};
    \node[red, align=center] at (\x2,\ya-0.85) {\scriptsize Deceptive policy \\ \scriptsize synthesis};
\end{tikzpicture}
\subcaption{}
\end{subfigure}
\begin{subfigure}{0.25\textwidth}
\tikzmath{\y1 = 2.3; \y2 = 1.6; \y3 = 0.9; \xt1 = \x1 +0.2; \xlab1 = \xt1 - 0.25;}
\begin{tikzpicture}
    \node (s0) at (\xlab1,\y2) {$s_0$};

    \node (10) at (\xt1,\y1) {};
    \node (11) at (\x3,\y1) {};
    \node (11d) at (\x3,\y1-0.1) {};
    \node (20) at (\xt1,\y2) {};
    \node (21) at (\x3,\y2-0.3) {};
    \node (21d) at (\x3,\y2+0.3) {};
    \node (30) at (\xt1,\y3) {};
    \node (31) at (\x3,\y3) {};

    \node[red] (12) at (\x2,\y2 + 0.7) {\scriptsize $\pi^\mathcal{A}_i$};
    \node[blue] (13) at (\x2,\y2 - 0.5) {\scriptsize $\pi^\mathcal{S}_i$};
    \draw [in=10, out=30,red] (20.center) to (11.center) ;
    \draw [in=20, out=20,blue,densely dotted] (20.center) to (11d.center) ;
    
    \draw [in=190, out=20,red] (20.center) to (21d.center) ;
    \draw [in=80, out=0,blue,densely dotted] (20.center) to (21.center) ;
    
    \draw [in=120, out=-10,blue,densely dotted] (20.center) to (31.center) ;
    \draw [in=90, out=-15,red] (20.center) to (31.center) ;

    \draw[rounded corners=0.5cm,draw opacity=0] (\xc - \bRecW,0) rectangle (\xc + \bRecW,1.3) {};
    \draw[rounded corners=0.5cm,draw opacity=0] (\xc - \bRecW,1.35) rectangle (\xc + \bRecW,2.65) {};
    \draw[rounded corners=0.5cm,draw opacity=0] (\xc - \bRecW,2.7) rectangle (\xc + \bRecW,\bRecH) {};
    \node[align=center,red] at (\x2,\ya-0.8) {\scriptsize Policy \\ \scriptsize Execution};
\end{tikzpicture}
\subcaption{}
\end{subfigure}
\begin{subfigure}{\sfSep\textwidth}
\tikzmath{
    \xa2 = 0.6;
    \x1 = \xc - \xa2;
    \x2 = \xc;
    \x3 = \xc + \xa2;
}
\begin{tikzpicture}
    \node[blue] at (\x2,\ys-0.6) {\scriptsize Keep};
    \node[red] at (\x2,\ya-0.2-0.6) {\scriptsize Eliminate};
    \path (\x1,\ys) node[circle,draw,inner sep=2pt,blue](1) {\scriptsize $\pi^\mathcal{S}_1$}
            (\x2,\ya-0.2) node[circle,draw,inner sep=2pt,red](2) {\scriptsize  $\pi^\mathcal{A}_2$}
            (\x3,\ys) node[circle,draw,inner sep=2pt,blue](3) {\scriptsize  $\pi^\mathcal{S}_3$};
    \draw[rounded corners=0.5cm, rectangle split parts=2,draw opacity=0] (\xc - \bRecW,0) rectangle (\xc + \bRecW,\bRecH) {};
\end{tikzpicture}
\subcaption{}
\end{subfigure}
\caption{a) Supervisor assigns reference policies. b) Agents decide on deceptive policies. c) Agents execute their policies in the environment. d) Supervisor eliminates a subset of agents based on observed behavior.}
\label{fig:ProcessOutline}
\end{figure}
We model each agent with a Markov decision process (MDP), and we define success for the team as any agent reaching the goal in their MDP.
In particular, certain states in each MDP represent the agents' shared reachability task. The agents'
deceptive policies then must satisfy the constraint that, with high probability, at least one agent reaches a target state. 
For example, in a surveillance task, only one agent must deviate to obtain footage of a secure location.

The agents need centralized synthesis to complete their shared task with high probability, but we limit the agents to decentralized policies to improve the tractability of synthesis and remove the need for communication. The use of decentralized policies is a shotgun approach. Each agent follows a policy independently from the other agents and achieves the task with a small probability, but collectively, the agents achieve the task with high probability.

We use Kullback-Leibler (KL) divergence, as often used in security settings \cite{bai2017data,kanellopoulos2021bounded,khazraei2022resiliency}, to measure deviations between agent behavior and the reference policy.
In stochastic environments, paths that achieve the agents' task may be feasible under the reference policy, but the likelihood of the paths informs the supervisor's belief about whether an agent is deceptive. 
The agents can make their paths plausible under the reference policy by ensuring KL divergence is small.

We study two versions of deceptive policy synthesis, which differ in how the agents 
avoid elimination and whether the agents use decoys.
We first formulate deceptive policy synthesis 
as ensuring all agents' deviations are small. In particular, we formulate \textit{\AgProbSimple} as minimizing the worst KL divergence among all agents.
We then explore settings in which, after the supervisor eliminates some agents, others may complete the task. 
By choosing policies carefully, one may allocate decoy agents, which the supervisor eliminates.
In \textit{\AgProbFull}, we formulate the supervisor's elimination procedure and explore the synthesis of policies such that when the supervisor eliminates decoy agents, the remaining agents succeed.

We give efficient algorithms for each synthesis problem.
The shotgun approach we take leads to a non-convex reachability constraint for these problems. However, we give an efficient method to find globally optimal solutions to \textit{\AgProbSimple} via a sequence of convex optimization problems. We also discuss how the supervisor may use this algorithm to increase system security by improving their reference policies.
We then explore restrictions to the supervisor's elimination procedure to make \textit{\AgProbFull} more tractable. We solve this problem by extending the algorithm provided for \textit{\AgProbSimple} to allocate decoy agents. 

\subsection{Related Work}
We discuss several areas relevant to our work, including team deception, KL divergence in security, deception of observers, and decentralized MDPs.

\textbf{Team Deception.}
Various disciplines study application-specific team deception.
Examples include the clustering of deceivers in online games \cite{yu2015} or the use of decoy agents to aid a leader in misdirecting a multi-robot team \cite{pettinati2021wolves}.
In contrast to application-specific approaches, we explore the synthesis of deceptive policies for a team of agents represented by MDPs.
Existing approaches for team deception include mean-field approaches \cite{chen2022deceptive} and reinforcement learning \cite{ghiya2020learning}. In contrast, we explore optimization-based approaches in a non-mean-field regime. Furthermore, these works focus on the problem of obscuring a task while we explore the concealment of the policy used to achieve a task.
Prior work also explores secure multi-agent planning \cite{he2024security,mu2023quantified,shi2023security,yu2022security}. 
These works represent security with opacity-like formulations, where an observer must not be able to determine that the agents have visited some state.
Finally, hidden role games are team games where agents are unaware of the team composition. 
Existing literature explores equilibrium computation \cite{carminati2023hidden} and learning-based approaches \cite{aitchison2021learning,serrino2019finding,strouse2018learning}.

\textbf{KL Divergence and Security.}
Most relevant to our setting is the synthesis of deceptive policies in MDPs using KL divergence \cite{karabag2021}.
Deceptive policy synthesis via KL divergence minimization admits a convex formulation with dimension polynomial in the MDP size. 
One may also formulate similar KL divergence minimization problems for partially observable agent dynamics \cite{karabag2022}, continuous dynamics \cite{patil2023}, and stochastic games \cite{karabag2024}.
Our work contrasts \cite{karabag2021,karabag2022,patil2023} with the addition of multiple supervised agents. One could use the formulation in \cite{karabag2021} to explore multi-agent settings, but the resulting implementation would be intractable for many agents and would require communication between agents.
The shotgun approach we use, with decentralized execution, is more tractable and does not need communication.
More generally, KL divergence appears in analyzing attack detection \cite{bai2017data,kanellopoulos2021bounded,khazraei2022resiliency}. 
For example, in the context of input replacement attacks in a linear system, KL divergence relates to an attack's stealthiness \cite{bai2017data}.

\textbf{Deception of Observers.}
Various works explore formulating policies to mask agents' intent from observers in single-agent settings. Quantitative deception literature includes approaches based on minimizing KL divergence, i.e. expected log-likelihood ratio, \cite{karabag2021,karabag2022,patil2023}, as well as
constraining the probability of the log-likelihood ratio exceeding a threshold \cite{ma2023covert}.
Meanwhile, in qualitative intention deception, an attacker ensures that observations generated by their behavior are consistent with observations generated by non-deceptive agents \cite{fu2022almost}. Again, in contrast to \cite{fu2022almost,ma2023covert}, we consider multiple observed agents.
Deceptive path planning also involves an agent masking their intent from an adversary by finding paths that delay an observer's recognition of the agent's goal \cite{fatemi2024deceptive,masters2017deceptive}. 
We consider a distinct problem from deceptive path planning, as in our setting, the agents obscure the decision-making process used to reach a state rather than the state itself.
Finally, the likelihood ratio between the paths produced by hidden Markov models defines the form of probabilistic opacity considered in \cite{keroglou2018probabilistic} for verification. This work is relevant to our work as we synthesize policies to control the log-likelihood ratio of observations produced by two Markov chains.

\textbf{Centralized Planning and Decentralized Execution.}
Our work relates to 
decentralized execution approaches common in multi-agent learning and planning.
For example, multi-agent reinforcement learning may use centralized learning with decentralized execution \cite{lowe2017multi}.
In planning, decentralized Markov decision processes (Dec-MDPs) are most relevant to our setting. Solving Dec-MDPs is NEXP-complete \cite{bernstein2002complexity} in general and NP-complete with independent transitions \cite{becker2004solving}. However, some classes of Dec-MDPs, such as Dec-MDPs with additive rewards and shared additive resource constraints, have efficient solution methods \cite{de2021constrained}. 
Additionally, heuristic methods provide good solutions for chance-constrained problems with additive rewards \cite{undurti2011decentralized}. We explore a chance-constrained problem where the reward has a maximum structure rather than an additive structure, and we show this maximum structure allows globally optimal policy synthesis.

\section{Preliminaries}
\label{sec:Preliminaries}
For $n$ objects $a_i$ indexed by $i=1,\ldots,n$, the collection is $(a_i)_{i=1}^n$. The set $[n]$ contains the natural numbers $1,\ldots,n$.
For probability distributions $P_1,P_2$ with a support $\mathcal{X}$, the KL divergence is
    $\textrm{KL}(P_1||P_2) = \sum_{x\in\mathcal{X}} P_1(x) \log \left( \nicefrac{P_1(x)}{P_2(x)} \right)$.

\textbf{Markov Decision Processes.}
A Markov decision process (MDP) $\mathcal{M}$ is a tuple $(S,A,P,s_0)$ where $S$ and $A$ are state and action spaces, $P$ is a transition function, and $s_0$ is an initial state. The set of actions available at state $s$ is $A(s)$, and
the probability of transitioning from state $s$ to $q$ with action $a$ is $P({s,a,q})$.
The set of successor states of $s$, $\textrm{Succ}(s)$, contains states $q$ such that there exists an action $a \in A(s)$ with $P({s,a,q}) > 0$.
For an absorbing state, $\textrm{Succ}(s) = \{s\}$. 

A stationary policy is a map $\pi: S \times A \rightarrow [0,1]$ such that $\sum_{a\in A(s)} \pi(s,a) = 1$ for all $s \in S$.
For an MDP $\mathcal{M}$, $\Pi(\mathcal{M})$ is the set of stationary policies on $\mathcal{M}$. The Cartesian product of these sets for $n$ MDPs is $\boldsymbol{\Pi}(\MDP{i}) = \Pi(\MDP{1}) \times \ldots \times \Pi(\MDP{n})$. 
Note that $\boldsymbol{\Pi}(\MDP{i})$ contains tuples of stationary policies, rather than policies on the joint state.
A path in an MDP with policy $\pi$ is a state-sequence $\xi = s_0s_1\ldots$ such that, for all $t$, $\sum_{a\in A(s_t)}P({s_t,a,s_{t+1}})\pi(s_t,a) > 0$. 
If each of $n$ MDPs runs for $\NumRounds$ rounds, the $j_{th}$ path of MDP $\MDP{i}$ is $\xi_{i,j}$. The sequence of paths from $\MDP{i}$ is $\boldsymbol{\xi}_i = (\xi_{i,j})_{j=1}^{\NumRounds}$.
A stationary policy $\pi$ induces a distribution $\Gamma^\pi$ on the paths of an MDP, and
the KL divergence between policies $\pi_1$ and $\pi_2$ is 
$\KL{\pi_1}{\pi_2}$.

For an MDP and stationary policy, the occupancy measure of state-action pair $(s,a)$ is $x_{s,a} = \sum_{t=0}^\infty \Pr(s_t=s|s_0)\pi(s,a)$, and is the expected number of visits to $(s,a)$.
By an abuse of notation, $x_{s,q} = \sum_{a\in A(s)}x_{s,a}P({s,a,q})$ is the occupancy flow from state $s$ to $q$. Similarly, $\smash{\pi_{s,q}=\sum_{a\in A(s)}P({s,a,q})\pi(s,a)}$ is the probability of transitioning from state $s$ to $q$ under policy $\pi$.

For a single agent,
$\Pr\left(s_0 \models \Diamond R \right)$ is the probability of reaching set $R$ with the agent's policy.
For a set $T$ of agents, with policies $\pi_i$,
$\Pr\left(\exists i \in T: s_0^i \models \Diamond R_i \right)$ is the probability that at least one agent reaches set $R_i$ in their MDP. We refer to this probability as the disjunctive reachability probability, and we remark that we may compute this probability from the agents' independent failure probabilities using $\Pr\left(\exists i \in T: s_0^i \models \Diamond R_i \right) = 1 - \prod_{i=1}^n \left(1 - \Pr\left(s_0^i \models \Diamond R_i \right)\right)$.

\section{Problem Formulation}
\label{sec:ProblemFormulation}

We first discuss the problem setting and then formulate two synthesis problems. 

For $i = 1, \ldots ,n$, an MDP $\MDP{i}$ governs agent $i$.
The supervisor assigns each agent a stationary policy $\Pol{S}{i}$.
The agents have a shared disjunctive reachability task, and they achieve this task if any agent reaches set $\ReachSet{A}{i}\subseteq S_i$ in $\MDP{i}$.
The agents choose policies $\Pol{A}{i}$ in a centralized manner such that $\Pr(\exists i\in [n]:  s_0^{i}\models \Diamond R^\mathcal{A}_i ) \geq \nu_\mathcal{A}$. 
We assume $s$ is absorbing for all $s\in \bigcup_{i=1}^{n}R^\mathcal{A}_i$.

\noindent
\begin{minipage}{.54\textwidth}
\noindent\textbf{Running Example.}
We give an aerial package delivery example for ease of exposition in Figure~\ref{fig:RunningExample}.
Two agents navigate the state space $ S = \{1,2,3,4,* \}$ with actions $ r$ (right) and $ d$ (down). However, due to weather, the agent may not go in the commanded direction. Additionally, weather may force a landing at state $*$ when taking action $r$ at state $2$.
A $\textrm{land}$ action is also available, which transitions the agent from state $2$ to state $*$ with probability $1$.
Agent $i$'s target is $R_i^\mathcal{A} = \{*\}$ for $i\in\{1,2\}$. 
\end{minipage}%
\hfill
\tikzmath{\ax = 1.9 ;
\x1 = 0; \x2 = \x1 + \ax; \x3 = \x2  + 1.3;
            \y1 = 1.2; \y2 = 0;}
\begin{minipage}{.42\textwidth}
\begin{tikzpicture}[shorten >=1pt,node distance=2cm,auto,initial text=,]
  \node[state,initial,inner sep=2pt, minimum size=16pt]  (q_1) at (\x1,\y1)                     {$1$};
  \node[state,inner sep=2pt, minimum size=16pt]          (q_2) at (\x2,\y1) {$2$};
  \node[state,inner sep=2pt, minimum size=16pt]          (q_*) at (\x2,\y2) {$*$};
  \node[state,inner sep=2pt, minimum size=16pt]          (q_3) at (\x3,\y1) {$3$};
  \node[state,inner sep=2pt, minimum size=16pt]          (q_4) at (\x1,\y2) {$4$};

  \path[->] (q_1) edge              node [below]       {\scriptsize $d,0.1$} (q_2)
                  edge [bend left]  node  {\scriptsize $r,0.9$}  (q_2)
                  edge              node         {\scriptsize $r,0.1$} (q_4)
                  edge [bend right]  node [left]  {\scriptsize $d,0.9$} (q_4)
            (q_2) edge              node        {\scriptsize $r,0.8$} (q_3)
                  edge              node        {\scriptsize $\textrm{land},1$} (q_*)
                  edge [bend right]  node [left]      {\scriptsize $r,0.2$} (q_*);
\end{tikzpicture}

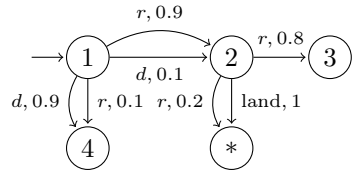
\captionof{figure}{An MDP for a package delivery example. Circles are states, and arrows are transitions. An arrow labeled $a,P({s,a,q})$, from state $s$ to $q$, represents an action.}
\label{fig:RunningExample}
\end{minipage}

\subsection{Worst-Case Deceptive Policy Synthesis}

To make the team deceptive, we ensure all agents have low deviation, and we quantify deviation with KL divergence using hypothesis testing arguments as in \cite{karabag2021}.
The log-likelihood test is the most powerful hypothesis test for a given significance level \cite{cover1999}, and we assume the supervisor eliminates an agent $i$ based on this test after observing their path in the state space. 
Although the paths are infinite, they terminate in some absorbing state, and we take a conservative approach by assuming the supervisor observes this whole path.
For a null hypothesis that agent $i$ is not deceptive, i.e. $\pi_i = \pi^\mathcal{S}_i$, and an alternative hypothesis that agent $i$ is deceptive, i.e. $\pi_i = \pi^\mathcal{A}_i$, the KL divergence between $\pi^\mathcal{A}_i$ and $\pi^\mathcal{S}_i$ gives the expected log-likelihood ratio. A policy with a lower KL divergence then corresponds to a less detectable agent and a smaller deviation.

We define the optimal deceptive policies as minimizing the maximum KL divergence, thus ensuring that all agents have low deviation. We formulate the agents' problem as finding optimal policies subject to reachability constraints.
\begin{problem}
    (\textit{\AgProbSimpleTitle})
    Given 
     MDP $\MDP{i}$, reference policy $\Pol{S}{i}$, set $R^\mathcal{A}_i$ for each agent $i\in[n]$, and threshold $\nu_\mathcal{A} \geq 0$, solve
    \begin{subequations}
        \begin{align}
             \underset{(\Pol{A}{1},\ldots,\Pol{A}{n})\in \boldsymbol{\Pi}(\MDP{i})}{\inf} \quad & \underset{i \in [n]}{\max} \ \KL{\Pol{A}{i}}{\Pol{S}{i}}\\
            \textrm{subject to} \ \quad \quad & \Pr(\exists i\in [n]:  s_0^{i}\models \Diamond R^\mathcal{A}_i ) \geq \nu_\mathcal{A}. \label{subeq:ReachConst}
        \end{align}
    \end{subequations}
    \label{prob:AgentSimple}
\end{problem}

The shotgun approach we use, with decentralized policies and centralized planning, is beneficial as it leads to a scalable solution and does not need inter-agent communication.
One may also consider both centralized policies and planning to make the deceptive team as a whole less detectable.
One achieves this approach by using existing deceptive planning methods \cite{karabag2021} on a joint MDP, with state and action space formed by products of individual state and action spaces.
However, this approach has two issues. 
First, the approach does not scale, as the joint MDP grows exponentially with $n$.
Also, centralized policies need communication during execution, which may be unavailable.
We avoid these issues by using decentralized execution and centralized planning, as commonly used in multi-agent reinforcement learning \cite{lowe2017multi}. Practically, a single party
who knows the deceptive agents' identities finds the policies and transmits them to the agents.

\noindent\textbf{Running Example (continued).} Consider reference policies given by $\Pol{S}{1}(1,r)$ $ = \Pol{S}{1}(2,r) =  \Pol{S}{2}(1,d) = \Pol{S}{2}(2,r) = 1.$
Under these policies, $\Pr(s_0^1 \models \Diamond \{*\} \vee s_0^2 \models \Diamond \{*\}) = 0.2$.
For $\nu_A = 0.5$, the agents must deviate to satisfy their specification. An example of a feasible deviation is $(\Pol{A}{1})(1,r) = (\Pol{A}{1})(2,\textrm{land}) = 1$, and $\Pol{A}{2} = \Pol{S}{2}$. The combined probability of success is then $0.9$. 

\subsection{Elimination-Aware Deceptive Policy Synthesis}
\label{sec:BeliefFormulation}

We explore the synthesis of deceptive policies such that, after a supervisor eliminates some agents, the remaining agents complete the task with high probability.
In some settings, the supervisor eliminates a subset of agents after an observation period, and the agents may want to ensure that the remaining agents still achieve the task with high probability.
Alternatively, an agent's success may be contingent on the supervisor not detecting that agent.
To explore these settings, we model the supervisor's procedure for eliminating agents, and we introduce a parameter to measure the supervisor's elimination budget. 
We then define the agents' problem as maximizing the budget for which the remaining agents complete the task with high probability. 

\subsubsection{A Supervisor Elimination Procedure.} 
For a prior $\DecPri{i} = \DecPriShort{i}$, the supervisor computes belief $\Belief{i} = \Pr(\isNotDec{i}|\boldsymbol{\xi}_i)$ using
\begin{equation}
    \label{eq:SupervisorBayesUpdate}
    \Belief{i} = 1 - \frac{\DecPriShort{i}}{\DecPriShort{i} + (1-\DecPriShort{i}) \left(\nicefrac{\NonDecCond{i}}{\DecCond{i}}\right)}.
\end{equation}
As $\Belief{i}$ increases, the supervisor is more sure agent $i$ is not deceptive.

We define the supervisor's procedure for eliminating agents as a subset selection problem with parameters based on $\Belief{i}$.
The expected utility of agent $i$ to the supervisor is $V_i\Belief{i}$, where $V_i$ is a base utility. 
We then define the supervisor's budget for eliminating agents by an upper bound, $C$, on the total utility the supervisor may eliminate. Under this constraint, the supervisor maximizes the probability that at least one of the agents eliminated is deceptive, $\Pr(\exists i \in T: \isDec{i}|\boldsymbol{\xi}_i)$. 
Assuming the supervisor treats the agents as independent, conditioned on paths, $\Pr(\exists i \in T: \isDec{i}|\boldsymbol{\xi}_i) = 1 - \prod_{i\in T} \Belief{i}$. Maximizing $\Pr(\exists i \in T: \isDec{i}|\boldsymbol{\xi}_i)$ is then equivalent to maximizing $\sum_{i\in T}-\log (\Belief{i})$, and the supervisor's elimination problem is
\begin{equation}
    \label{eq:SupervisorSelectionMath}
     \underset{T}{\max} \quad \quad \sum_{i\in T}-\log (\Belief{i}) \quad
     \textrm{subject to} \quad \sum_{i\in T} \Belief{i}V_i \leq C.
\end{equation}
We denote the solution of \eqref{eq:SupervisorSelectionMath} by $T = f(\Belief{i}, C)$. 
The supervisor computes $T$ once after observing $\NumRounds$ paths for the agents.

We remark that the independence assumptions affect the form of (3). 
The conditional independence assumption applies when the supervisor treats the agents' deceptiveness as independent, for example, when the deceptive agents, $1,\ldots,n$, are a subset of a larger agent pool that the supervisor observes. 

\subsubsection{Deceptive Policy Synthesis with Elimination.}
We define the agents' problem as synthesizing policies that maximize the value of $C$, such that the remaining agents, $[n] \setminus f(\Belief{i}, C)$, complete the task with high probability.

Policies alone do not determine the subset $T$, as elimination depends on the state sequence produced during execution. As such, we define a belief proxy
\begin{equation}
    \BeliefProxy{i}{\pi} = 1 - \frac{\DecPriShort{i}}{\DecPriShort{i} + (1-\DecPriShort{i})e^{-\NumRounds \cdot \KL{\pi}{\Pol{S}{i}}}}.
\end{equation}
Because of the relationship,
$$\NumRounds \cdot 
\KL{\pi}{\Pol{S}{}} = \mathbb{E} \left[ \log (\Pr (\boldsymbol{\xi}|\pi) ) - \log (\Pr (\boldsymbol{\xi}|\pi^\mathcal{S}))\right],$$
the belief proxy has the same form as the expectation $\mathbb{E}[\theta_i(\boldsymbol{\xi}_i)]$, with the expectation operator moved to the exponent in the denominator.

Using this proxy, we formulate the agents' problem as 
\begin{problem}
    \label{prob:DecOptimizationGeneral}
    (\textit{\AgProbFullTitle})
    Given MDP $\MDP{i}$, reference policy $\Pol{S}{i}$, set $R^\mathcal{A}_i$, prior $\DecPriShort{i}\in(0,1)$, and base utility $V_i \geq 0$ for each agent $i\in[n]$, as well as threshold $\nu_\mathcal{A} \geq 0$, and number of paths $\NumRounds \in \mathbb{N}$, solve
    \begin{subequations}
        \begin{align}
            \underset{(\Pol{A}{1},\ldots,\Pol{A}{n})\in\boldsymbol{\Pi}(\MDP{i}),C,T}{\sup} \quad \ & C\\
            \textrm{subject to} \quad \quad \quad \ & T =  f(\BeliefProxy{i}{\pi^\mathcal{A}_i}, C), \\
            & \Pr (\exists i\in [n]\setminus T:  s_0^{i}\models \Diamond R^\mathcal{A}_i ) \geq \nu_\mathcal{A}. \label{subeq:reachOptGen}
        \end{align}
    \end{subequations}
\end{problem} 
For $V_1=\ldots=V_n$, and fixed policies $\Pol{A}{i}$, the supremum of the set of feasible values of $C$ is the smallest amount of utility the supervisor must sacrifice such that the remaining agents no longer satisfy the reachability constraint.

\noindent\textbf{Running Example (continued).}  
Consider policies formulated for Problem~\ref{prob:AgentSimple} such that constraint (\ref{subeq:ReachConst}) is tight. If $\BeliefProxyTEXT{i}{\Pol{A}{i}} = y$, and $V_i = 1$, for all $i\in [n]$, then for $C = y$, these policies are no longer feasible, as when the supervisor eliminates either agent, the task is no longer achieved with high probability.
Alternatively, each agent may reach $*$ with probability $\nu_\mathcal{A}$ at the cost of decreasing $\BeliefProxyTEXT{i}{\Pol{A}{i}}$.
If $\BeliefProxyTEXT{i}{\Pol{A}{i}} = z$ for all $i
\in [n]$, then the policies are feasible for any $C < 2z$, as the supervisor must eliminate both agents to violate the reachability constraint.

\section{Worst-Case Deceptive Policy Synthesis}
\label{sec:AgentPolicyConstructionSimple}
We provide a scalable algorithm to solve Problem~\ref{prob:AgentSimple} to global optimality by solving a sequence of convex optimization problems for each agent individually.

We first convert Problem~\ref{prob:AgentSimple} into a formulation based on occupancy measures to facilitate the reachability constraint (\ref{subeq:ReachConst}). This conversion follows a similar process to \cite{karabag2021}, which we detail in Section~\ref{sec:occMeasureChecks} of the appendix.
The reference policy $\pi^\mathcal{S}_i$ induces a set of transient states, $\sd{i} \subseteq S_i$, on which $\Pol{A}{i}$ deviates from $\Pol{S}{i}$.
The states in $ S_i \setminus \sd{i} $ are closed, and the optimal deceptive policies do not deviate from $\pi^\mathcal{S}_i$ on $S_i \setminus \sd{i}$. We optimize occupancy measures for $s \in \sd{i}$, and the elements of vector $\mathbf{x}_i$ are occupancy measures $x_{s^i,a^i}$ for the states $s^i \in \sd{i}$.

We define the following functions to introduce the new formulation. 
\begin{equation}
    \label{eq:KLOccupancy}
     \textrm{KL}(\mathbf{x}_i, \pi^\mathcal{S}_i) = \sum_{s^i\in \sd{i}} \sum_{q^i\in \textrm{Succ}_i(s^i)} x_{s^i,q^i} \\\log \left( \frac{x_{s^i,q^i}}{\pi_{s^i,q^i}^{\mathcal{S}}\sum_{a^{i'}\in A_i(s^i)} x_{s^i,a^{i'}}} \right).
\end{equation}
\begin{equation}
\label{eq:OccupancyFlow}
    F(\mathbf{x}_i, s^i) = \sum_{a^i\in A_i(s^i)}x_{s^i,a^i} - \sum_{q^i\in \sd{i}} x_{q^i,s^i} - \bbbone_{s_0^i}(s^i).
\end{equation}
\begin{equation}
\label{eq:Reachability}
    \nu(\mathbf{x}_i, R_i) = \sum_{q^i\in R_i} \sum_{s^i\in \sd{i}} x_{s^i,q^i} + \bbbone_{s_0^i}(q^i).
\end{equation}
The KL divergence between path distributions is (\ref{eq:KLOccupancy}), and (\ref{eq:KLOccupancy}) holds due to the stationarity of the policies \cite{karabag2021}.
The net occupancy flow at state $s^i$ is (\ref{eq:OccupancyFlow}), and the reachability probability for a set $R_i$ is (\ref{eq:Reachability}).

We reformulate Problem~\ref{prob:AgentSimple} as the following optimization problem, where the decision variables are the agents' individual occupancy measures.
\begin{subequations}
    \label{eq:OccupancyProblemInfMax}
    \begin{align}
         \underset{(\mathbf{x}_1^\mathcal{A}, \ldots ,\mathbf{x}_n^\mathcal{A})}{\inf} \quad \quad & \underset{i\in [n]}{\max} \ \textrm{KL}(\mathbf{x}_i^\mathcal{A}, \Pol{S}{i})\label{subeq:FiniteMax}\\
         \textrm{subject to} \quad 
         & F(\mathbf{x}_i^\mathcal{A}, s) = 0, \quad \forall s \in \sd{i},  \forall i\in[n], \label{subeq:FiniteOccupancy1}\\ 
         & \prod_{i=1}^{n} (1 - \nu(\mathbf{x}_i^\mathcal{A}, R^\mathcal{A}_i)) \leq 1 - \nu_\mathcal{A},  \label{eq:DisjunctiveReachability}
         \\
         & \mathbf{x}_i^\mathcal{A} \geq 0, \quad \forall i\in[n] \label{subeq:FiniteOccupancy2}.
    \end{align}
\end{subequations}
Proposition~\ref{prop:OptimizationWellPosedness} shows the 
existence of a solution and the equivalence of (\ref{eq:OccupancyProblemInfMax}) to Problem~\ref{prob:AgentSimple}. We give the proof in Section~\ref{sec:occMeasureChecks} of the appendix.
\begin{proposition}
    \label{prop:OptimizationWellPosedness}
    Problem~\ref{prob:AgentSimple} and the optimization problem (\ref{eq:OccupancyProblemInfMax}) share the same optimal value, and there exist policies $(\Pol{A}{i})_{i=1}^n$ that attain the optimal value.
\end{proposition}

We remark that (\ref{eq:DisjunctiveReachability}) is problematic as it induces non-convexity, and non-convex optimization problems may have sub-optimal local minima. 
However, Theorem~\ref{the:locGlobMin} shows that non-convexity is not an issue for Problem~\ref{prob:AgentSimple}.
\begin{theorem}
    \label{the:locGlobMin}
    Every local minimum of (\ref{eq:OccupancyProblemInfMax}) is a global minimum.
\end{theorem}
Theorem~\ref{the:locGlobMin} holds as (\ref{subeq:FiniteMax}) is the maximum of a finite set of convex functions, and the disjunctive reachability probability is a coordinate-wise monotone function. In fact, Theorem~\ref{the:locGlobMin} and the following algorithm hold for any objective that is a maximum of convex functions of individual occupancy measures. For example, one may instead minimize the maximum required battery capacity for a fleet of drones. We prove Theorem~\ref{the:locGlobMin} in Section~\ref{sec:OptimizationProofs} of the appendix.

Algorithm~\ref{alg:DeceptiveLineSearch} uses the maximum structure of (\ref{subeq:FiniteMax}) to solve (\ref{eq:OccupancyProblemInfMax}) via a series of reachability maximization problems.
The reachability maximization problem for agent $i$, given KL divergence bound $K$, reference policy $\Pol{S}{i}$, and set $R^\mathcal{A}_i$, is
\begin{subequations}
    \label{eq:SingleAgentSubProblem}
    \begin{align}
         \underset{\mathbf{x}^\mathcal{A}_i}{\sup} \quad \quad & \nu(\mathbf{x}^\mathcal{A}_i, R^\mathcal{A}_i) \\
         \textrm{subject to} \quad & \textrm{KL}(\mathbf{x}^\mathcal{A}_i, \Pol{S}{i}) \leq K, \label{subeq:ReachKLConstraint}\\
         & F(\mathbf{x}^\mathcal{A}_i, s^i) = 0, \quad \forall s^i \in \sd{i},  \\ 
         & \mathbf{x}^\mathcal{A}_i \geq 0 .
    \end{align}
\end{subequations}
We denote by $\textproc{Reach}(\Pol{S}{i}, R^\mathcal{A}_i, K)$, the optimal value of (\ref{eq:SingleAgentSubProblem}), and we note that (\ref{eq:SingleAgentSubProblem}) is a convex optimization problem due to results in \cite{karabag2021}. 
Algorithm~\ref{alg:DeceptiveLineSearch} finds the minimum $K$ such that the disjunctive reachability probability exceeds $\nu_\mathcal{A}$.
\begin{algorithm}
\caption{Line search applied for deceptive policy synthesis}
\label{alg:DeceptiveLineSearch}
\begin{algorithmic}[1]
\Procedure{DeceptiveSynthesis}{$(\Pol{S}{i},R^\mathcal{A}_i)_{i=1}^n ,\nu_A,K_{\textrm{max}},\varepsilon$}
    \State $\overline{K} \gets$ \Call{Bisection}{\textproc{ReachEvaluate}($(\Pol{S}{i},R^\mathcal{A}_i)_{i=1}^n,\nu_\mathcal{A},\cdot$),$[0,K_\textrm{max}]$,$\varepsilon$}
\EndProcedure
\Procedure{ReachEvaluate}{$(\Pol{S}{i},R^\mathcal{A}_i)_{i=1}^n,\nu_\mathcal{A},K$}
    \State $\nu \gets 1 - \prod_{i=1}^n (1 - \textproc{Reach}(\Pol{S}{i}, R^\mathcal{A}_i, K))$ \label{algLine:ProductAlt}
    \State \Return $\nu - \nu_\mathcal{A}$
\EndProcedure
\end{algorithmic}
\end{algorithm}

Algorithm~\ref{alg:DeceptiveLineSearch} finds the smallest zero crossing of $\textproc{ReachEvaluate}$.
$\textproc{Bisection}$ successively computes intervals $[\underline{K},\overline{K}]$ containing $K^*$, which is the smallest $K$ such that 
$\textproc{ReachEvaluate}((\Pol{S}{i},R^\mathcal{A}_i)_{i=1}^n,\nu_\mathcal{A},K) \geq 0$.
We can then use the final value of $\overline{K}$ with $\textproc{ReachEvaluate}$ to compute feasible policies. 
Note that $\textproc{Bisection}$ should not terminate when $\textproc{ReachEvaluate}$ is zero. Rather, we should continue decreasing $\overline{K}$ to ensure we find $K^*$.

Algorithm~\ref{alg:DeceptiveLineSearch} requires $\mathcal{O}(\log_2(K_{\max}/\varepsilon))$ iterations each of which requires solving $\mathcal{O}(n)$ single-agent problems.

We next describe how to find $K_{\max}$, which upper-bounds the optimal value of (\ref{eq:OccupancyProblemInfMax}).
For each $i$, construct a new MDP, $\hat{\MDP{i}}$, by removing actions that induce state transitions 
$(s^i,q^i)$ with zero probability under $\Pol{S}{i}$. 
For each $\hat{\MDP{i}}$, find policies $\Pol{A}{i}$ that maximize $\nu_i = \Pr(s_0^i \models \Diamond R^\mathcal{A}_i)$. If $1 - \prod_{i=1}^n (1 - \nu_i) \geq \nu_\mathcal{A}$, then the maximum KL divergence among the policies bounds $K^*$. 
If the inequality does not hold, the agents 
must use state transitions with zero probability under $\Pol{S}{i}$.

As $\textproc{ReachEvaluate}$ is monotonic in $K$, 
Algorithm~\ref{alg:DeceptiveLineSearch} converges to the optimal value of (\ref{eq:OccupancyProblemInfMax}) for finite $K_{\max}$. We give a proof in Section~\ref{sec:OptimizationProofs} of the appendix.
\begin{theorem}
    \label{the:LineSearchConvergenceBasic}
    The value $\overline{K}$ computed by Algorithm~\ref{alg:DeceptiveLineSearch}  satisfies $\overline{K} < K^* + \varepsilon$, where $K^*$ is the optimal solution of Problem~\ref{prob:AgentSimple}.
\end{theorem}
\subsection{On Reference Policy Synthesis}

The supervisor may preempt deceptive policy synthesis by choosing reference policies $\Pol{S}{i}$ that maximize the optimal value of Problem~\ref{prob:AgentSimple}. Denoting Problem~\ref{prob:AgentSimple}'s optimal value by 
$g((\pi_i^\mathcal{S})_{i=1}^n)$
, the supervisor's problem is
\begin{subequations}
    \label{eq:SupProbShort}
    \begin{align}                                       
    \underset{(\pi_1^\mathcal{S},\ldots,\pi_n^\mathcal{S}) \in \boldsymbol{\Pi}(\MDP{i})}{\sup} \quad & g((\pi_i^\mathcal{S})_{i=1}^n) \\
         \textrm{subject to} \quad \quad & \Pr(s_0^{i}\models \Diamond R^\mathcal{S}_i ) \geq \nu_{\mathcal{S},i},\quad \forall i \in [n]. \label{subeq:SupTask}
    \end{align}
\end{subequations}
In (\ref{subeq:SupTask}), $\ReachSet{S}{i}\subseteq S_i$ is a task for each agent, and $\nu_{\mathcal{S},i}$ is a probability threshold.

While solving (\ref{eq:SupProbShort}) in the single-agent case is NP-hard \cite{karabag2021}, \cite{karabag2023decision} uses projected gradient-descent as a heuristic.
Extending this approach to multiple agents is non-trivial.
The disjunctive reachability constraint underlying $g$ is non-convex, and this non-convexity makes taking projections in the agents' variables difficult. 

However, Algorithm~\ref{alg:DeceptiveLineSearch} facilitates first-order methods for (\ref{eq:SupProbShort}). The gradient descent with max-oracle algorithm \cite{lin2019} may solve a $\max-\min$ problem by taking gradient steps in the outer variables and solving the inner problem at each iteration.
In (\ref{eq:SupProbShort}), the inner problem is deceptive policy synthesis, which we solve with Algorithm~\ref{alg:DeceptiveLineSearch}. Note that we must smooth $\max_i\textrm{KL}(\mathbf{x}_i^\mathcal{A}, \Pol{S}{i})$ for differentiability.

\section{Elimination-Aware Deceptive Policy Synthesis}
\label{sec:DeceptiveWithElimination}

We explore the synthesis of deceptive policies that ensure agents complete the reachability task, even when the supervisor eliminates some agents.
We discuss the complexity of the supervisor's subset selection problem, as it appears in this synthesis problem, and we give restrictions that ease policy synthesis.
We then extend the methods for \textit{\AgProbSimple}, to solve Problem~\ref{prob:DecOptimizationGeneral}.

\subsection{Discussion of Supervisor's Elimination Procedure}
The supervisor's elimination procedure appears to lack sufficient structure to facilitate an efficient algorithm for deceptive policy synthesis.
A tuple of weights, $(w_i)_{i=1}^n$, and profits, $(p_i)_{i=1}^n$, define a knapsack problem
\begin{subequations}
    \label{eq:KnapsackGeneral}
    \begin{align}
     \underset{T\subseteq[n]}{\max} \quad \quad \sum_{i\in T}p_i \quad 
     \textrm{subject to} \quad \sum_{i\in T} w_i \leq C,
\end{align}
\end{subequations}
and the elimination procedure, (\ref{eq:SupervisorSelectionMath}), is an example with $p_i = -\log(\Belief{i})$, and $w_i = \Belief{i}V_i$.
Proposition~\ref{prop:knapsack} indicates that 
the supervisor's elimination procedure is too general to permit an efficient algorithm for deceptive policy synthesis, as instances of the elimination procedure cover real-valued knapsack problems.
\begin{proposition}
    \label{prop:knapsack}
    Let $(w_i)_{i=1}^n$ and $(p_i)_{i=1}^n$ be given arbitrary tuples of positive real weights and profits defining a knapsack problem. Then, there exists a tuple, $(\pi^\mathcal{S}_i, \pi^\mathcal{A}_i, \MDP{i}, \boldsymbol{\xi}_i, V_i,\DecPriShort{i})_{i=1}^n$, of reference policies, agent policies, MDPs, state paths, base utilities, and priors, such that $w_i = \Belief{i}V_i$ and $p_i = -\log(\Belief{i})$.
\end{proposition}
\begin{proof}
Consider an MDP with $S = \{o,a,b\}$, $A = \{1,2\}$, and $s_0 = o$. States $a$ and $b$ are absorbing, and $P({o,1,a}) = P({o,2,b}) = 1$. For all $i$, $\MDP{i} = (S,A,P,s_0)$. 
For all $i$, fix paths as $\boldsymbol{\xi}_i = (0,a,a,a,\ldots)$, and set $\DecPriShort{i} =\kappa \in (0,1)$. Note that $\NumRounds = 1$. We now design policies $\Pol{A}{i}$ and $\Pol{S}{i}$. 
Likelihood ratios are given by 
\begin{equation}
    \frac{\NonDecCond{i}}{\DecCond{i}} = \frac{\Pol{S}{i}(o,1)}{\Pol{A}{i}(o,1)},
\end{equation}
and we set $\nicefrac{\Pol{S}{i}(o,1)}{\Pol{A}{i}(o,1)} = \nicefrac{e^{-p_i}\kappa}{((1-e^{-p_i})(1-\kappa))}$ so that 
\begin{equation}
    \Belief{i} = 1 - \frac{\kappa}{\kappa + (1-\kappa)\left(\nicefrac{\Pol{S}{i}(o,1)}{\Pol{A}{i}(o,1)}\right)} = \exp(-p_i).
    \label{eq:CoveringProofCond}
\end{equation}
As $V_i$ is free, and $\Belief{i} > 0$, we set $V_i$ such that $w_i = \Belief{i}V_i$.
$\square$
\end{proof}
In Problem~\ref{prob:DecOptimizationGeneral}, the function $f$, which represents the supervisor's procedure for eliminating agents, takes belief proxies $\hat{\theta}$ as input, rather than true beliefs $\theta$.
However, we may prove Proposition~\ref{prop:knapsack} for $\hat{\theta}$ as well. We simply control $e^{-\textrm{KL}_i}$ instead of $\nicefrac{\Pol{S}{i}(o,1)}{\Pol{A}{i}(o,1)}$. It is easy to modify the policy pairs to set $\textrm{KL}_i$ to any value in $[0,\infty)$, and $e^{-\textrm{KL}_i}$ to any value on $(0,1]$. We then need to choose $\kappa$ small enough so that $\nicefrac{e^{-p_i}\kappa}{((1-e^{-p_i})(1-\kappa))} \in (0,1)$ for all $i\in[n]$.

Problem~\ref{prob:DecOptimizationGeneral} is also a bi-level knapsack problem, with the agents as the leader and the supervisor as the follower, and the complexity of bi-level knapsack problems \cite{pferschy2019,pferschy2021} reinforces the difficulty of Problem~\ref{prob:DecOptimizationGeneral}.
For example, if a leader controls weights in a knapsack problem solved by a follower, and the leader gets rewards based on the items selected, the problem of finding the optimal weights is not approximable \cite{pferschy2019}. 
While Problem~\ref{prob:DecOptimizationGeneral} is distinct in form from problems explored in \cite{pferschy2019,pferschy2021}, the hardness results further suggest the difficulty of Problem~\ref{prob:DecOptimizationGeneral}.

We restrict Problem~\ref{prob:DecOptimizationGeneral} such that 
$V_1 = \ldots = V_n = 1$, leading to a simple solution for the supervisor's problem.
With this restriction, the supervisor solves
\begin{equation} 
     \max_T \quad \sum_{i\in T}-\log (\Belief{i}) \quad 
     \textrm{subject to} \quad \sum_{i\in T} \Belief{i} \leq C, \label{subeq:resultsKS}
\end{equation}
to find a subset $T$ of agents to eliminate, and we may construct
the solution by adding agents to $T$ in increasing order of $\Belief{i}$ until we violate the constraint in (\ref{subeq:resultsKS}).
The interpretation of $V_i = V_j$ is that agents $i$ and $j$ have the same expected utility to the supervisor if $i$ and $j$ are equally believed to be non-deceptive. If $V_i > V_j$, the supervisor may prefer to eliminate agent $j$ even if $\Belief{i} < \Belief{j}$.

\subsection{Synthesis of Optimal Deceptive Policies under Elimination}
We reformulate Problem~\ref{prob:DecOptimizationGeneral} (\textit{\AgProbFull}) to apply the methods used for Problem~\ref{prob:AgentSimple}.

We first replace $f$, which defines the eliminated agents, with constraints defining the optimal subset $T$. This replacement yields the optimization problem
\begin{subequations}
    \label{eq:DecOptimizationReformI}
    \begin{align}
        \underset{(\pi_1^\mathcal{A},\ldots,\pi_n^\mathcal{A})\in\boldsymbol{\Pi}(\MDP{i}),T,C}{\max} \ & C\\
        \textrm{subject to} \quad \quad \ & \ \BeliefProxy{i}{\Pol{A}{i}} <                                \BeliefProxy{j}{\Pol{A}{j}}, \quad  \forall i\in T, \forall j\in [n]\setminus T, \label{subeq:KS3} \\
        & \sum_{i\in T} \BeliefProxy{i}{\Pol{A}{i}} \leq C,\label{subeq:KS1} \\
        & \sum_{i\in T} \BeliefProxy{i}{\Pol{A}{i}} + \BeliefProxy{j}{\Pol{A}{j}}> C, \quad \forall j \in [n]\setminus T, \label{subeq:KS2} \\
        & \Pr (\exists i\in [n]\setminus T:  s_0^{i}\models \Diamond R^\mathcal{A}_i ) \geq \nu_\mathcal{A}.
    \end{align}
\end{subequations}
As $V_1 = \ldots = V_n$, the set $T$ contains the $|T|$ agents with the lowest values of $\hat{\theta}$, and constraint (\ref{subeq:KS3}) encodes this fact. Constraint (\ref{subeq:KS1}) then enforces the knapsack constraint from (\ref{subeq:resultsKS}). Finally, constraint (\ref{subeq:KS2}) 
ensures that $T$ is optimal, as if we add any agent to $T$, then we violate the constraint in (\ref{subeq:resultsKS}). We note that (\ref{subeq:KS3}) restricts the set of feasible policies such that the agents in $T$ and $[n]\setminus T$ may not have the same value of $\smash{\hat{\theta}}$. 

We may manipulate $\hat{\theta}$ to force the elimination of certain decoy agents first, 
 and we simplify (\ref{eq:DecOptimizationReformI}) by considering this interpretation of $T$ as containing decoys.
For fixed $T$ and $\Pol{A}{i}$, 
    $ C = \sum_{i\in T} \BeliefProxy{i}{\Pol{A}{i}} + \min_{i\in [n]\setminus T} \ \BeliefProxy{i}{\Pol{A}{i}} - \delta$ 
is optimal, where we add $\delta \approx 0$ due to the strict inequality in (\ref{subeq:KS2}). 
This value of $C$ corresponds to the supervisor eliminating all decoy agents $i\in T$ and almost having the budget to eliminate the non-decoy agent with the lowest belief, $\min_{i\in [n]\setminus T} \ \BeliefProxyTEXT{i}{\Pol{A}{i}}$.
The decoys should also have maximum expected utility to the supervisor, subject to the constraint that their expected utility is lower than that of the non-decoy agents.
As $V_1 = \ldots = V_n$, we may set
$\Pol{A}{i}$ such that $\BeliefProxyTEXT{i}{\Pol{A}{i}}$= $\gamma \min_{j\in [n]\setminus T} \ \BeliefProxyTEXT{j}{\Pol{A}{j}}$ for all $i\in T$.   
The scalar $\gamma \in (0,1)$ accounts for the strict inequality in (\ref{subeq:KS3}).

Applying these equalities gives the optimization problem
\begin{subequations}
    \label{eq:DecOptimizationReformII}
    \begin{align}
        \underset{(\pi_1^\mathcal{A},\ldots,\pi_n^\mathcal{A})\in\boldsymbol{\Pi}(\MDP{i}),T,M}{\max} \quad & |T|(M \cdot \gamma) + M - \delta\\
        \textrm{subject to}
        \quad \quad \quad & \Pr (\exists i\in [n]\setminus T:  s_0^{i}\models \Diamond R^\mathcal{A}_i ) \geq \nu_\mathcal{A},\\
        & \ \BeliefProxy{i}{\Pol{A}{i}} \geq M, \quad \forall i\in [n]\setminus T,\\
        & \ \BeliefProxy{i}{\Pol{A}{i}} = M\cdot \gamma, 
        \quad \forall i \in T.
    \end{align}
\end{subequations}
To solve (\ref{eq:DecOptimizationReformII}), we sweep the size of $T$, and for each $|T| = k$, we optimize the decoy assignment and agent policies to maximize $\min_{i\in [n]\setminus T} \ \BeliefProxyTEXT{i}{\Pol{A}{i}}$.

For fixed $|T|$,
we reformulate (\ref{eq:DecOptimizationReformII}) by substituting KL divergence for belief proxies so that we may apply a similar line search procedure to Algorithm~\ref{alg:DeceptiveLineSearch}. 
Assuming $p_{\mathcal{D},1} = \ldots = p_{\mathcal{D},n}$, rearranging the definition of the belief proxy yields
\begin{equation}
    \BeliefProxy{i}{\Pol{A}{i}} > \BeliefProxy{j}{\Pol{A}{j}} \ \textrm{if and only if} \ \KL{\Pol{A}{i}}{\Pol{S}{i}}  < \KL{\Pol{A}{j}}{\Pol{S}{j}}.
\end{equation}
As such, we may equivalently minimize the maximum KL divergence among the $n-k$ agents in $[n]\setminus T$, which yields
\begin{subequations}
    \label{eq:FinalFormDecGeneral}
    \begin{align}
        \underset{(\pi_1^\mathcal{A},\ldots,\pi_n^\mathcal{A})\in\boldsymbol{\Pi}(\MDP{i}),T}{\min} \ & \max_{i\in [n]\setminus T} 
        \KL{\Pol{A}{i}}{\Pol{S}{i}} \\
        \textrm{subject to} \quad \quad
        & \   \KL{\Pol{A}{i}}{\Pol{S}{i}} = K\cdot \gamma', \quad \forall i\in T, \label{subeq:Decoy}\\
        & \ \Pr (\exists i\in [n]\setminus T:  s_0^{i}\models \Diamond R^\mathcal{A}_i ) \geq \nu_\mathcal{A}, \\
        & \ |T| = k.
    \end{align}
\end{subequations}
We replace tolerances $\gamma$ in $\BeliefProxy{i}{\Pol{A}{i}}$ with tolerances in KL divergence $\gamma' \in (1,\infty)$. Tuning $\gamma'$ controls how much more decoys deviate compared to non-decoys. 
We remark that while the non-linear equality constraint (\ref{subeq:Decoy}) appears difficult to satisfy, we can find the required policy with convex combinations in the policy space between $\pi_i^\mathcal{S}$ and a policy $\pi_i$ with KL divergence greater than $K \cdot \gamma'$.

Algorithm~\ref{alg:DeceptiveSubset} combines a modified version of the line search procedure from Algorithm~\ref{alg:DeceptiveLineSearch} with a sweep across subset sizes to solve (\ref{eq:FinalFormDecGeneral}). 
We modify \textproc{ReachEvaluate} by computing the disjunctive reachability using the $n-k$ agents that have maximum reachability probability. These $n-k$ agents are non-decoy agents.

\begin{algorithm}
\caption{Elimination-aware deceptive policy synthesis}
\label{alg:DeceptiveSubset}
\begin{algorithmic}[1]
\Procedure{DeceptiveSubsetSelection}{$(\pi^\mathcal{S}_i, R^\mathcal{A}_i)_{i=1}^n$, $\nu_\mathcal{A}$,$K_{\max}$,$\varepsilon$,$p$,$\gamma'$,$\NumRounds$}

\For{$k=0, \dots, n-1$}
    \State $B_k \gets 0$    \algorithmiccomment{Assume initially that $k$ decoys are not feasible}
    \State $K, Fail_k \gets $ \Call{SubsetSearch}{$(\pi^\mathcal{S}_i, \Diamond R^\mathcal{A}_i)_{i=1}^n$, $\nu_\mathcal{A}$,$K_{\max}$,$\varepsilon$, $n-k$} \label{line:FailPoint}
    \If{$\neg Fail_k$}
        \State $M' \gets 1 - \frac{p}{p + (1-p)e^{-\NumRounds K\gamma'}}$, $M \gets 1 - \frac{p}{p + (1-p)e^{-\NumRounds K}}, B_k \gets k\cdot M' + M \label{Line:BComp}$
    \EndIf
\EndFor
\State $k^* \gets \arg \max B_k$
\EndProcedure
\Procedure{SubsetSearch}{$(\pi^\mathcal{S}_i, R^\mathcal{A}_i)_{i=1}^n$, $\nu_\mathcal{A}$, $K_{\max}$,$\varepsilon$, $w$}
    \State $\overline{K} \gets $ \Call{Bisection}{\textproc{ReachEvaluateSub}($(\pi^\mathcal{S}_i,R^\mathcal{A}_i)_{i=1}^n$, $\nu_\mathcal{A}$,$\cdot$,$w$),$[0,K_{\max}]$}
    \State $\nu - \nu_\mathcal{A} \gets$ \Call{ReachEvaluateSub}{$(\pi^\mathcal{S}_i, R^\mathcal{A}_i)_{i=1}^n$, $\nu_\mathcal{A}$,$\overline{K}$,$w$}
    \State \Return $\overline{K}, (\nu - \nu_\mathcal{A} < 0)$ \algorithmiccomment{Validate whether the final policy satisfies reach.}
\EndProcedure
\Procedure{ReachEvaluateSub}{$(\pi^\mathcal{S}_i, R^\mathcal{A}_i)_{i=1}^n$, $\nu_\mathcal{A}$,$K$,$w$}
    \State $\nu_i \gets \textproc{Reach}( \pi^\mathcal{S}_i, R^\mathcal{A}_i, K), \quad \forall i\in[n]$
    \State $\mathcal{N} \gets \Call{BestKElements}{w,\{\nu_i\}}$ \algorithmiccomment{Get indices of $w$ highest $\nu_i$ values}\label{line:ComplementNonDecoy}
    \State $\nu \gets 1 - \prod_{i\in \mathcal{N}} (1 - \nu_i)$
    \State \Return $\nu - \nu_\mathcal{A}$ 
\EndProcedure
\end{algorithmic}
\end{algorithm}

We again give a method to set $K_{\max}$.
Compute $K_{\max}'$ using the process given in Section~\ref{sec:AgentPolicyConstructionSimple} for Algorithm~\ref{alg:DeceptiveLineSearch}, and compute $B_0$ using $K = K_{\max}'$ in Line~\ref{Line:BComp}. This $B_0$ value lower-bounds the optimal value of (\ref{eq:DecOptimizationReformII}). 
Then, compute $K_{\max}$ as the value such that, using $K=K_{\max}$ in Line~\ref{Line:BComp}, $B_{n-1} = B_0$. 
For $|T|=n-1$, if the non-decoy agents have KL divergence above $K_{\max}$, the solution is worse than the solution for $|T|=0$ with KL divergence of $K_{\max}'$.
In fact, for any $k$, the solution is sub-optimal if the non-decoy agents' KL divergence exceeds $K_{\max}$.
The $Fail_k$ flag is set if no solution exists with KL divergence below $K_{\max}$ for $|T| = k$.

We also assume that, for all $i\in[n]$, we may synthesize policies $\Pol{*}{i}$ with KL divergence value higher than the optimal value of (\ref{eq:FinalFormDecGeneral}), but we can also use Algorithm~\ref{alg:DeceptiveSubset} without this assumption.
We need this assumption so that any agent $i$ may be a decoy.
Without this assumption, we modify Line~\ref{line:ComplementNonDecoy}. 
If $K$ is above 
the maximum KL divergence that agent $i$ may attain, then $i$ can not be a decoy, and we include $i$ in $\mathcal{N}$ before adding agents to $\mathcal{N}$ based on their reachability.
\section{Numerical Results}
\label{sec:Results}

We illustrate deceptive policy synthesis in a delivery example, where weather-induced stochasticity provides plausible deniability for deceptive agents.
A supervisor specifies $\Pol{S}{i}$ for $n$ drones so that each drone delivers a package to some target. 
However, an external intruder with access to the drones must ensure that, with high probability, at least one drone delivers a package to their location.

We define the drones' MDPs and targets via an undirected graph $G = (V,E)$, where $V$ contains regions, and $E$ contains connections between regions.
The state space is $V\times\{0,1\}$, where a drone in state $(v,0)$ is in flight in $v$, and a drone in state $(v,1)$ has landed in $v$.
Given $N \subseteq V$, the agents' target states, for each $i\in[n]$, are $R_i^\mathcal{A} = \{(v,1) | v\in N\}$. 
For each $i\in[n]$, the supervisor's target state is $(n_i,1) $ for some $n_i\in V$. All landed states are absorbing.

The actions available to a drone at state $(v,0)$ are moving on the graph or landing.
If a drone uses action $a_u$, for node $u$ adjacent to $v$, the drone will transition to $(u,0)$ with probability $p_{target}$.
However, due to weather, the drone may transition to $(v,1)$ with probability $p_{land}$, or to an adjacent flight node $w\neq u$, with probability $\nicefrac{(1 - p_{land} - p_{target})}{(|\textrm{adj}(v)| - 1)}$. The set of nodes adjacent to $v$ is $\textrm{adj}(v)$.
When the drone takes a landing action at state $(v,0)$, the drone will transition to $(v,1)$ with probability $p_{target} + p_{land}$ and will transition to an adjacent flight state $(u,0)$ with probability $\nicefrac{(1 - p_{target} - p_{land})}{|\textrm{adj}(v)|}$.

We specify reference policies such that the agent $i$ moves along the shortest path in the graph $G$ toward the target node $n_i$. 

\subsection{Worst-Case Deceptive Policy Synthesis and Reference Policy Synthesis}
\begin{figure}[t]
    \centering
    \parbox{.93\textwidth}{
    \begin{subfigure}{0.3\textwidth}
        \includegraphics[width=\textwidth]{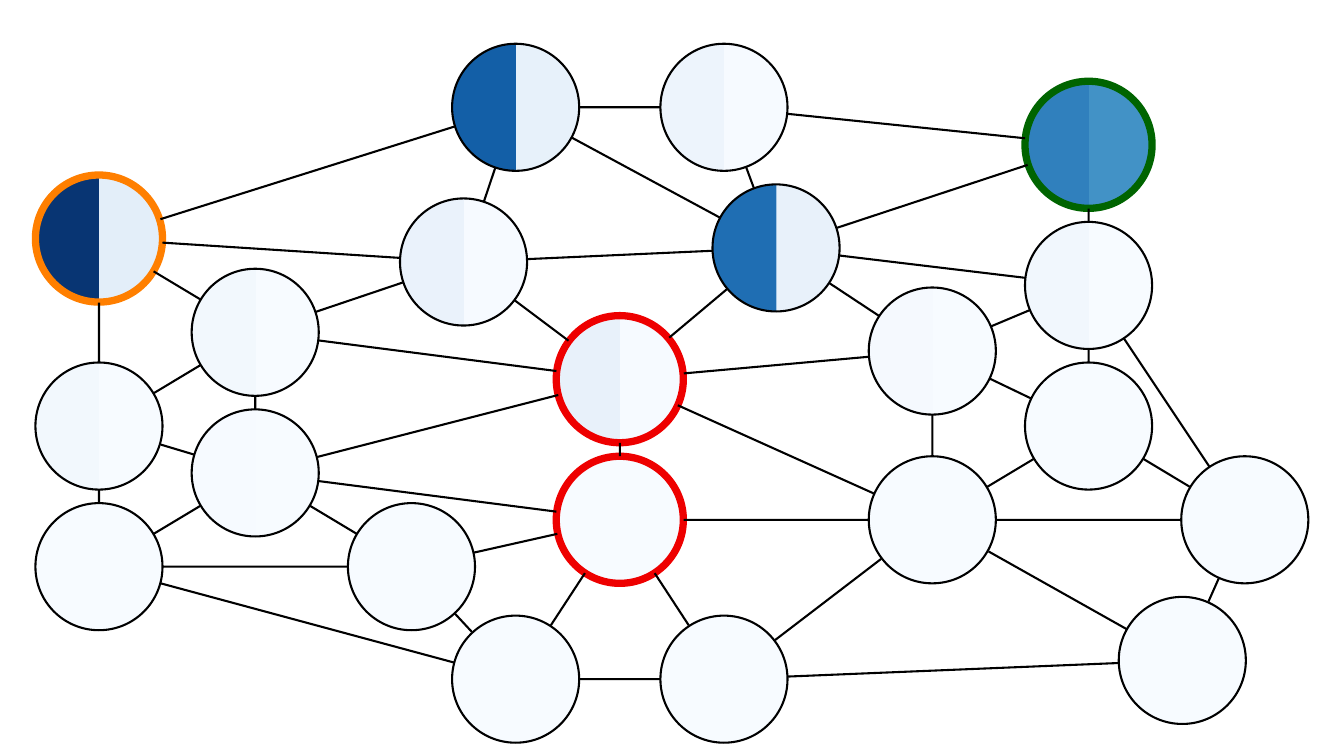} 
        \caption{$\Pol{S}{1}$}
        \label{fig:sup1}
    \end{subfigure}
    \hfill
    \begin{subfigure}{0.3\textwidth}
        \includegraphics[width=\textwidth]{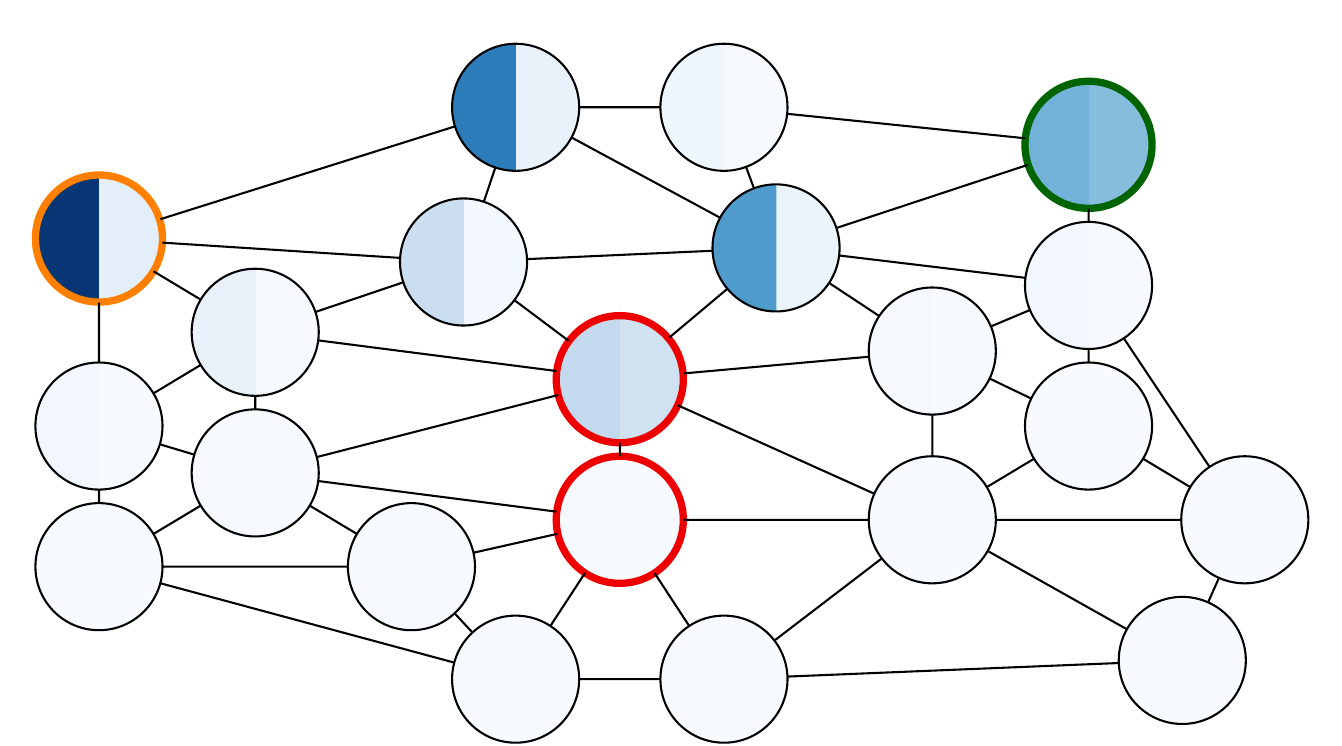} 
        \caption{$\Pol{A}{1}$}
        \label{fig:ag01}
    \end{subfigure}
    \hfill
    \begin{subfigure}{0.3\textwidth}
        \includegraphics[width=\textwidth]{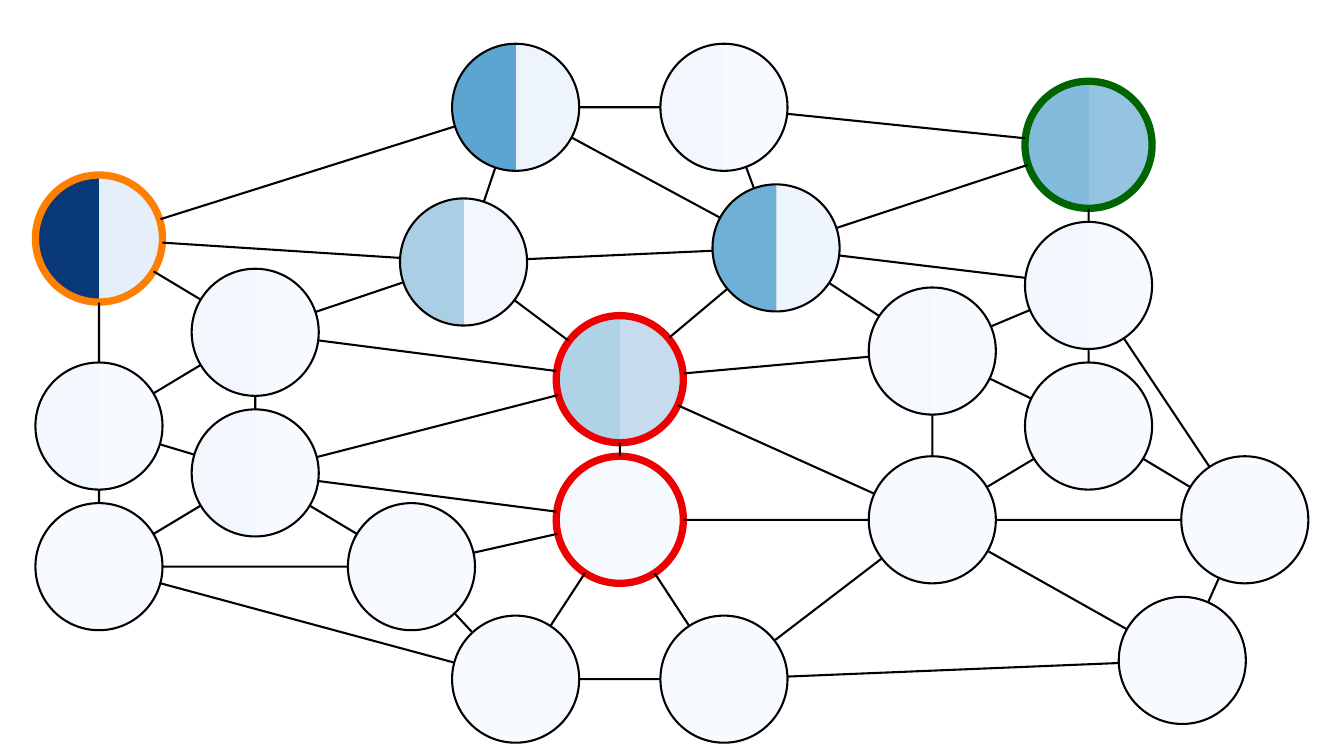} 
        \caption{Decoy $\Pol{A}{1}$}
        \label{fig:ag11}
    \end{subfigure}
    \begin{subfigure}{0.3\textwidth}
        \includegraphics[width=\textwidth]{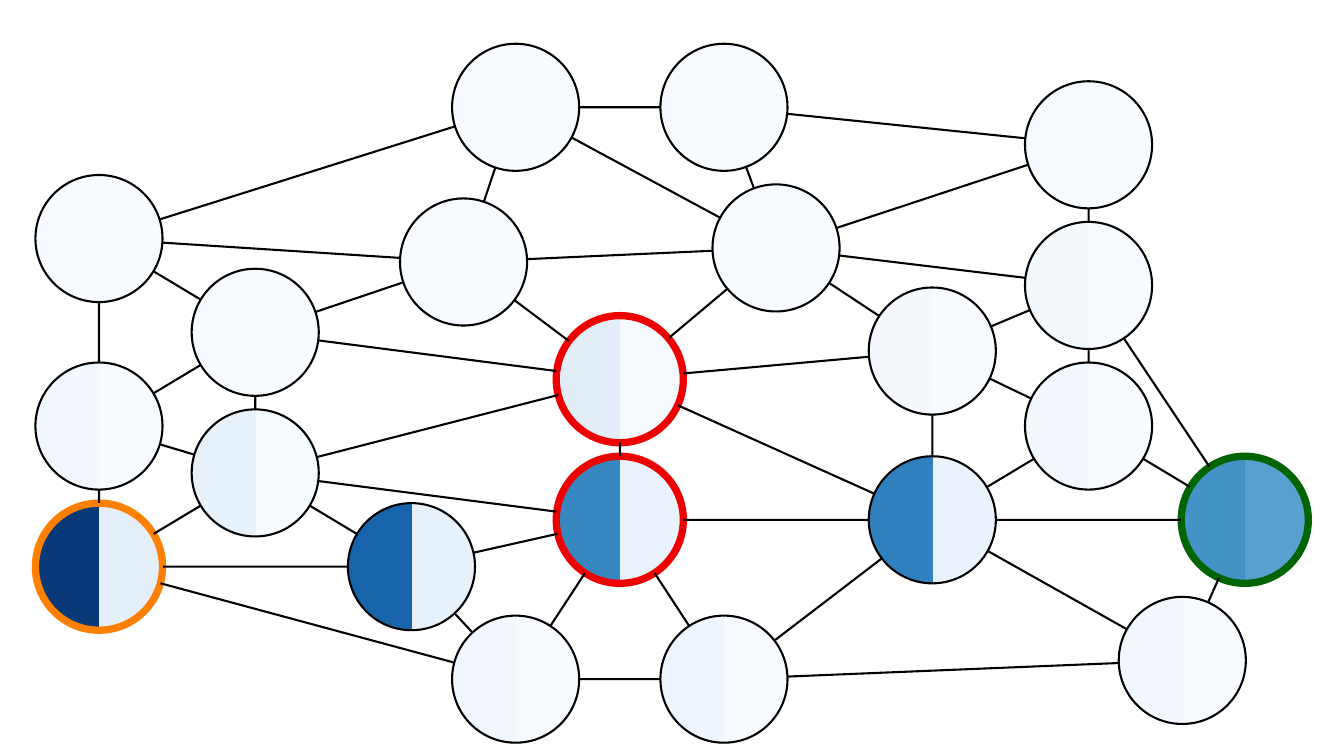}
        \caption{$\Pol{S}{2}$}
        \label{fig:sup2}
    \end{subfigure}
    \hfill
    \begin{subfigure}{0.3\textwidth}
        \includegraphics[width=\textwidth]{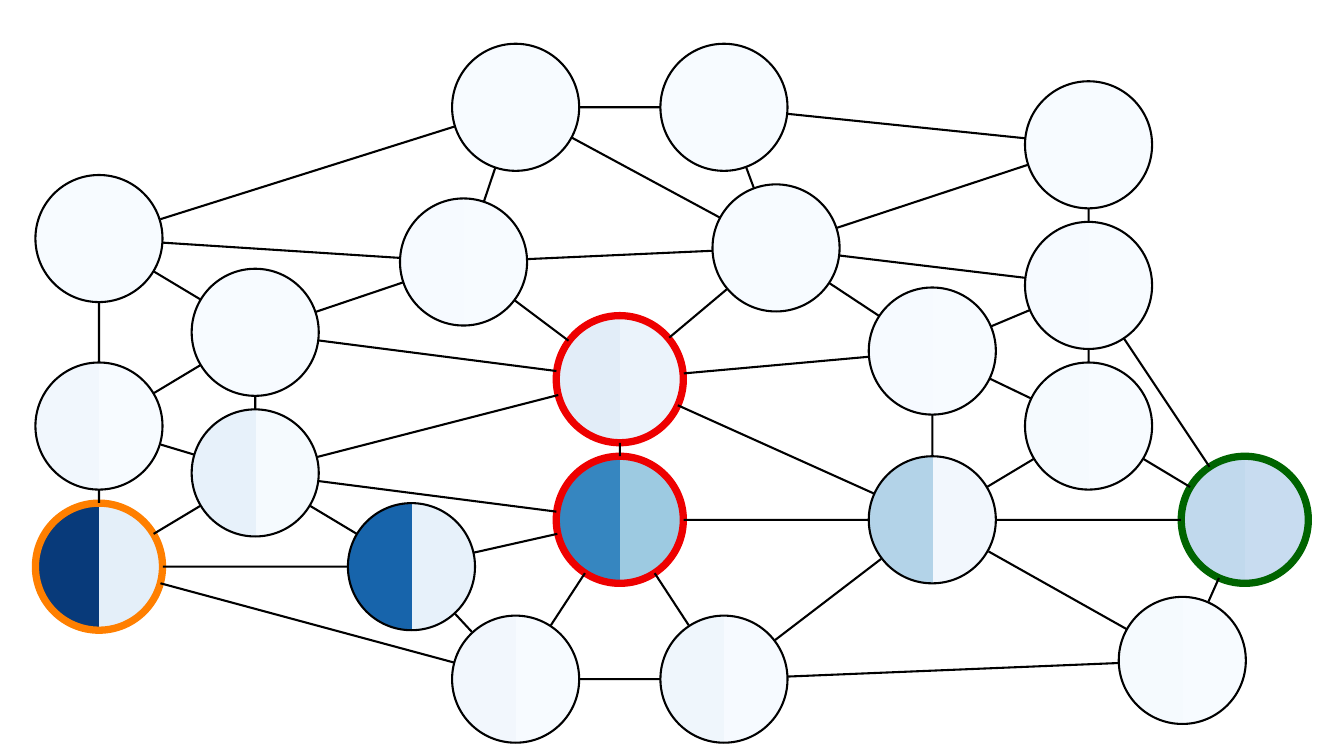} 
        \caption{$\Pol{A}{2}$}
        \label{fig:ag02}
    \end{subfigure}
    \hfill
    \begin{subfigure}{0.3\textwidth}
        \includegraphics[width=\textwidth]{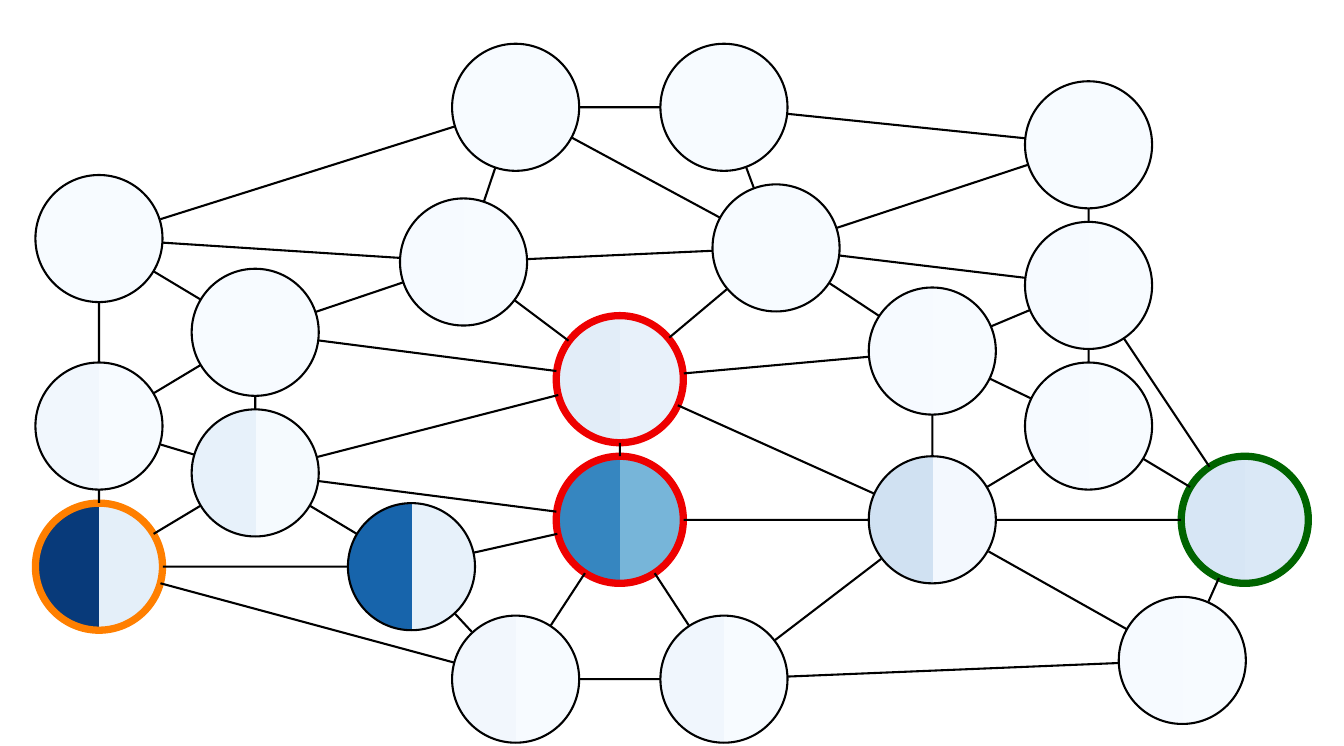} 
        \caption{Non-Decoy $\Pol{A}{2}$}
        \label{fig:ag12}
    \end{subfigure}
    }
    \hfill
    \begin{subfigure}[C]{0.06\textwidth}
        \includegraphics[width=\textwidth]{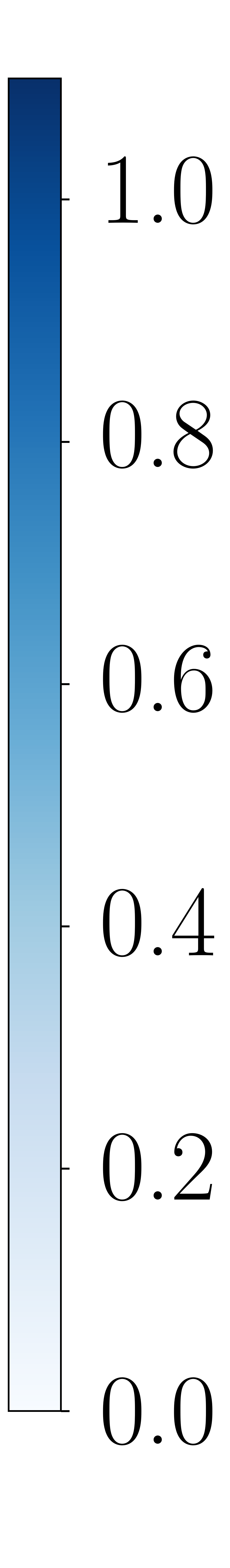} 
        \vfill
    \end{subfigure}
    \caption{Package delivery example policies.
    The left color of node $v$ is the occupancy measure of state $(v,0)$. The right color of a node $v$ is the occupancy flow into state $(v,1)$, which is equal to $\Pr(s_0^i \models \Diamond\{(v,1)\})$.
    Nodes with green, red, and orange borders are supervisor target nodes, agent target nodes, and initial states, respectively. 
    (a) and (d) are reference policies. (b) and (e) are deceptive policies for no decoys, and are policies synthesized by Algorithm~\ref{alg:DeceptiveLineSearch}. (c) and (f) show policies synthesized by Algorithm~\ref{alg:DeceptiveSubset} for one decoy agent with $\gamma' = 1.2$.} 
    \label{fig:DeceptiveSubsetNavExample}
\end{figure}
\newcommand\vSupi{0.009}
\newcommand\vSupii{0.09}
\newcommand\vAgenti{0.23}
\newcommand\vAgentii{0.48}
\newcommand\klAgent{0.59}
\newcommand\vAgentDeci{0.28}
\newcommand\vAgentDecii{0.6}
\newcommand\klAgentDeci{0.94}
\newcommand\klAgentDecii{1.12}

The first two columns of Figure~\ref{fig:DeceptiveSubsetNavExample} show
\textit{\AgProbSimple} \ for $n = 2$ and $\nu_\mathcal{A} = 0.6$, and we observe that the relative reachability probabilities for the agents match how suited their reference policies are to the task.
Under $\Pol{S}{2}$, agent $2$ already reaches $N\times\{1\}$ with probability \vSupii, as opposed to \vSupi \ for agent $1$, and so agent $2$ needs to deviate less to achieve higher reachability for the same KL divergence budget.
As such, under deceptive policies, agents $1$ and $2$ reach the target states with probabilities \vAgenti, and \vAgentii, respectively.

Figure~\ref{fig:RefPol} shows the reference policies we obtain by using Algorithm~\ref{alg:DeceptiveLineSearch} as a subroutine in the gradient ascent with max-oracle algorithm \cite{lin2019}.
As expected, this approach creates reference policies that avoid the agents' target states.
\subsection{Elimination-Aware Deceptive Policy Synthesis}
Figures~\ref{fig:ag11} and \ref{fig:ag12} show deceptive policies synthesized with one decoy, and we see that Algorithm~\ref{alg:DeceptiveSubset} allocates agent $1$ as the decoy as agent $1$ is less capable of completing the task under $\Pol{S}{1}$. 

We also consider an instance of the package delivery example with $n=8$ to demonstrate 
the trade-off between redundancy and detectability that arises when using decoys. In this example, the supervisor assigns $\Pol{S}{1}$ in Figure~\ref{fig:DeceptiveSubsetNavExample} to four agents and $\Pol{S}{2}$ to the other agents. We show the results in Figure~\ref{fig:BkChange}.
With no decoys, each agent needs a small reachability probability and has a low KL divergence, 
but the supervisor only needs to eliminate one agent for the reachability probability to drop below $0.6$.
Meanwhile, with seven decoys, all agents have large deviations, so the supervisor sacrifices little utility by eliminating all of the agents.
A mix of decoy and non-decoy agents maximizes the utility the supervisor must sacrifice so that the remaining agents complete the task with a probability below $0.6$.

\begin{figure}[t]
\begin{minipage}[b]{.58\textwidth}
    \captionsetup{type=figure}
    \centering
    \begin{minipage}{.49\textwidth}
    \includegraphics[width=\textwidth]{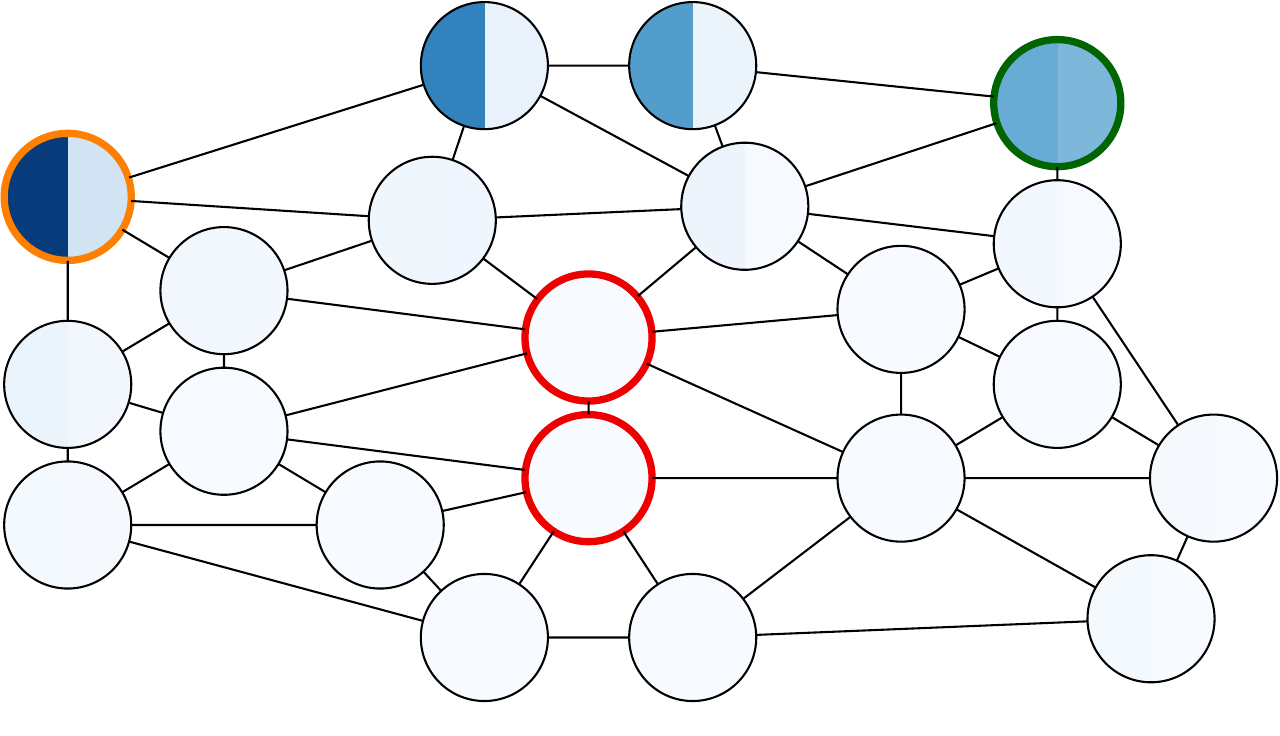}
    \centering
    (a) $\pi_1^{\mathcal{S}'}$
    \end{minipage}
    \begin{minipage}{0.49\textwidth}
    \includegraphics[width=\textwidth]{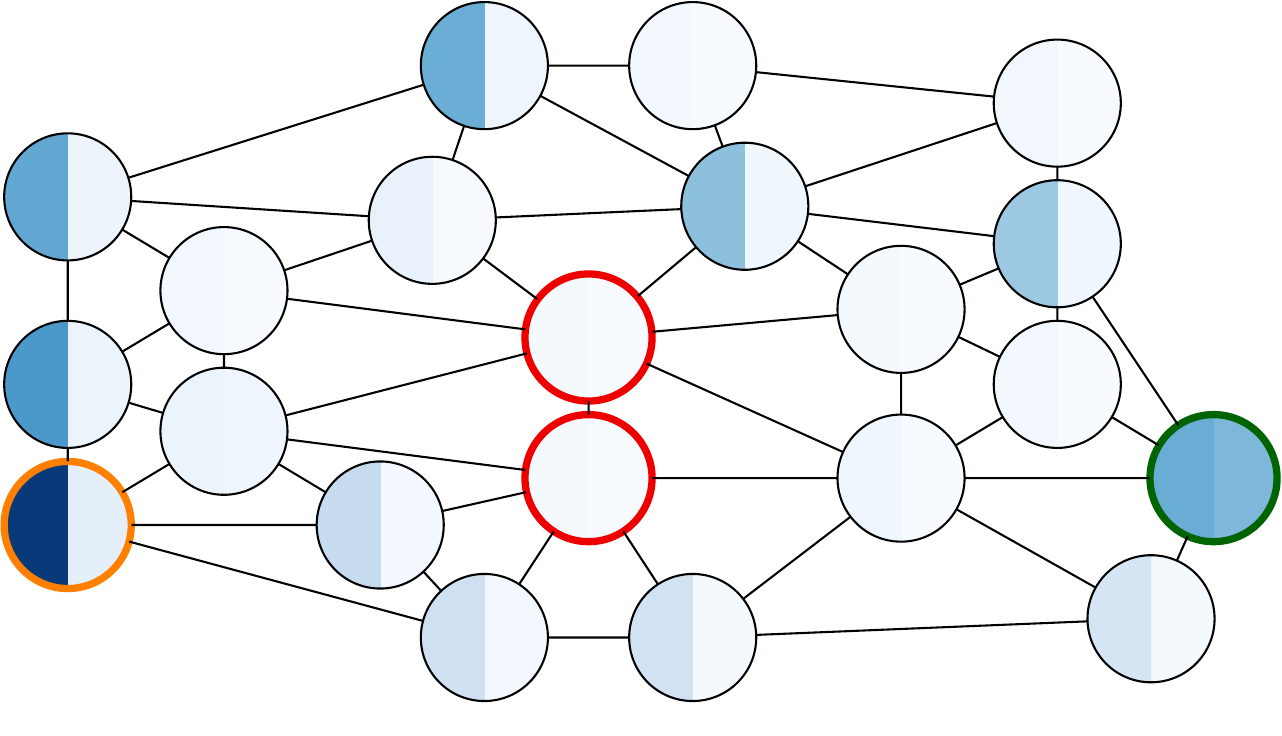}
    \centering
    (b) $\pi_2^{\mathcal{S}'}$
    \end{minipage}%
  \captionof{figure}{Synthesis of reference policies using gradient-ascent with max oracle with Algorithm~\ref{alg:DeceptiveLineSearch} used to solve the inner problem.} 
  \label{fig:RefPol}
\end{minipage}%
\hfill
\begin{minipage}[b]{.4\textwidth}
  \centering
  \begin{tikzpicture}
\begin{axis}[
    xlabel={Number of decoy agents},
    ylabel style={align=center},
    label style={font=\footnotesize},
    ylabel={Cost of eliminating \\decoy agents and one \\ non-decoy agent, $B_k$},
    xmin=0, xmax=7,
    ymin=0, ymax=1.7,
    xtick={0,1,2,3,4,5,6,7},
    ytick={0,0.4,0.8,1.2,1.6},
    legend pos=north west,
    ymajorgrids=true,
    grid style=dashed,
    height = 110pt,
    width = 110pt,
    tick label style={font=\footnotesize},
]
\addplot[
    only marks,
    color=blue,
    mark=asterisk,
    ]
    coordinates {
    (0,0.4081030340991759)(1,0.7677361279331695)(2,1.0842392570852315)(3,1.3434626349266914)(4,1.5262059253368303)(5,0.9428527211469809)(6,0.1583192736894975)(7,0.00017428565384280947)
    };
\end{axis}
\end{tikzpicture}
\captionof{figure}{Values of $B_k$ with $n=8$. }
  \label{fig:BkChange}
\end{minipage}%
\end{figure}
\section{Conclusion}
Building secure multi-agent systems requires system designers to consider team deception.
We explored the synthesis of deceptive policies for agents who deceive their supervisor by deviating from assigned reference behavior.
We formulated the agents' problem as finding decentralized policies that minimize the KL divergence between the agents' behavior and the reference while ensuring that, with high probability, at least one agent reaches a target.
While the formulation led to a non-convex optimization problem, we provided a scalable method to synthesize optimal policies.
We then demonstrated how a supervisor may use this method to improve the security of reference policies.
We also analyzed an extension of the problem where agents synthesize policies to ensure that after the supervisor eliminates some agents, the remaining agents complete the task.
This problem is difficult as the agents must reason about the subset selection procedure the supervisor uses to eliminate agents.
However, we explored a restriction to make the problem tractable and gave an algorithm for choosing deceptive policies that control the order in which the supervisor eliminates agents.

In this work, we explored disjunctive reachability, and we minimized the maximum deviation among the agents.
Future work will explore different couplings of the agents through their objectives and different measures of the team's detectability.

\section{Appendix}
\subsection{Sufficiency of Occupancy Measure Formulation}
\label{sec:occMeasureChecks}
We show the sufficiency of occupancy measure-based formulations for Problem~\ref{prob:AgentSimple}, following \cite{karabag2021}.
We require finiteness of the optimal value of Problem~\ref{prob:AgentSimple} for the following results, but we may check this finiteness by finding $K_{\max}$ with the method we describe in Section~\ref{sec:AgentPolicyConstructionSimple}. \

We first recall the state-space decomposition from \cite{karabag2021}.
Let $C^{cl}_i$ be the union of the closed communicating classes of the Markov chain induced by $(\MDP{i}, \Pol{S}{i})$. 
With one agent, the agent should follow $\Pol{S}{i}$ for all $s \in C^{cl}_i$ so that the policy does not have infinite KL divergence \cite{karabag2021}. The agent should also follow $\Pol{S}{i}$ on $R^\mathcal{A}_i$. These properties hold for Problem~\ref{prob:AgentSimple} as well, and set $\sd{i} = S_i \setminus (C^{cl}_i \cup R^\mathcal{A}_i)$ contains states on which we modify $\Pol{A}{i}$ from $\Pol{S}{i}$.

We now restate a result from \cite{karabag2021} on the finiteness of occupancy measure. Let Problem~\ref{prob:AgentSimple}a be the case of Problem~\ref{prob:AgentSimple} with a single-agent.
\begin{proposition}
    \label{prop:Karabag1}
    \cite{karabag2021}
    If Problem~\ref{prob:AgentSimple}a has finite optimal value, with optimal policy $\Pol{A}{}$, the state-action occupancy measure $x_{s,a}$ is finite for all $s \in S_d$ and $a \in A(s)$.
\end{proposition}
This result extends to Problem~\ref{prob:AgentSimple}.
Let Problem~\ref{prob:AgentSimple} have finite optimal value, with optimal policies $\Pol{A}{i}$, and define $\nu_i$ as $\Pr(s_0^i \models \Diamond \ReachSet{A}{i})$ under $\Pol{A}{i}$. Define P1i as
the single-agent problem of minimizing the KL divergence of $i$ with probability threshold $\nu_\mathcal{A}' = \nu_i$.
By construction, $\Pol{A}{i}$ is feasible for $\textrm{P1i}$, and its KL divergence is finite. Proposition~\ref{prop:Karabag1} then implies that solution $\pi^{\mathcal{A}*}_i$ to $\textrm{P1i}$ has finite occupancy measure on $\sd{i}$. As $\pi^{\mathcal{A}*}_i$ is optimal, it also has lower KL divergence than $\pi^{\mathcal{A}}_i$.
As such, for each agent $i\in[n]$, we replace $\Pol{A}{i}$ with $\pi^{\mathcal{A}*}_i$, to construct an optimal solution to Problem~\ref{prob:AgentSimple} comprised of policies with finite occupancy measure.

The following result also trivially extends to the problem we explore due to the decentralized approach we take.
\begin{proposition}
    \label{prop:Karabag2}
    \cite{karabag2021}
    For any policy $\Pol{A}{}$ that satisfies $\Pr(s_0\models  R^\mathcal{A}) \geq \nu_\mathcal{A}$, there exists a stationary policy $\pi^{\mathcal{A},St} \in \polSpace{}$ that satisfies 
    $\Pr(s_0\models  R^\mathcal{A}) \geq \nu_\mathcal{A}$ and 
    \begin{equation}
    \textrm{KL}\left(\Gamma^{\pi^{\mathcal{A},St}}||\Gamma^{\Pol{S}{}} \right) \leq \textrm{KL}\left(\Gamma^{\Pol{A}{}}||\Gamma^{\Pol{S}{}}\right).
    \end{equation}
\end{proposition}

\paragraph{Proof of Proposition~\ref{prop:OptimizationWellPosedness}.}
This proposition follows from the proof of the equivalence of Problem~\ref{prob:AgentSimple} and (\ref{eq:OccupancyProblemInfMax}) in the single-agent case, given in \cite{karabag2021}. 

The extensions of Propositions~\ref{prop:Karabag1} and \ref{prop:Karabag2} to multiple agents justify the restriction to stationary policies with finite occupancy measures for Problem~\ref{prob:AgentSimple}.

Regarding the existence of a policy that attains the optimal value of (\ref{eq:OccupancyProblemInfMax}), the only comment that needs to be made to extend the proof in \cite{karabag2021}, is that the objective 
    $\max_{i\in[n]} \ \textrm{KL}(\mathbf{x}_i,\Pol{S}{i})$
remains continuous in $\mathbf{x}_i$. $\square$

\subsection{Optimality and Algorithms for Problem~\ref{prob:AgentSimple}}
\label{sec:OptimizationProofs}
\paragraph{Proof of Theorem~\ref{the:locGlobMin}.} 
We show that any strictly sub-optimal point is not a local minimum. 
Let $(\mathbf{x}_i)_{i=1}^n$ be strictly sub-optimal, and let $(\mathbf{x}^*_i)_{i=1}^n$ be globally optimal. Let $\nu_i = \nu(\mathbf{x}_i, R^\mathcal{A}_i)$, $\nu^*_i = \nu(\mathbf{x}^*_i, R^\mathcal{A}_i)$, and $K_i(\mathbf{x}) = \textrm{KL}(\mathbf{x}, \pi^{\mathcal{S}}_i)$.

\textit{Case 1.} Assume constraint (\ref{eq:DisjunctiveReachability}) is not tight. 
Define new occupancy measures through the convex combination $\mathbf{x}'_i = \mathbf{x}_i + \theta (\mathbf{x}_i^* - \mathbf{x}_i)$ for all $i \in [n]$. 
Constraints (\ref{subeq:FiniteOccupancy1}) and (\ref{subeq:FiniteOccupancy2}) define convex sets, and so they will hold for $\mathbf{x}'_i$.
The left hand side of (\ref{eq:DisjunctiveReachability}) is continuous in $\mathbf{x}$, and constraint (\ref{eq:DisjunctiveReachability}) is not tight, so for sufficiently small $\theta > 0$ we have $\prod_{i=1}^{n} (1 - \nu(\mathbf{x}_i', R^\mathcal{A}_i)) \leq 1 - \nu_\mathcal{A}$.
Also, as $(\mathbf{x}_i)_{i=1}^n$ is strictly sub-optimal, we have that $\max_i \  K_i(\mathbf{x}_i) > \max_i \ K_i(\mathbf{x}_i^*)$, and hence
$\max_i \  K_i(\mathbf{x}_i) >\max_i \ K_i(\mathbf{x}_i')$ for $\theta > 0$. By the above observations, we may generate feasible $(\mathbf{x}_i')_{i=1}^n$ arbitrarily close to $(\mathbf{x}_i)_{i=1}^n$ with strictly lower objective value, which demonstrates that $(\mathbf{x}_i)_{i=1}^n$ cannot be a local minimum.

\textit{Case 2.} Now, allow constraint (\ref{eq:DisjunctiveReachability}) to hold with equality such that $\prod_{i=1}^{n} (1 - \nu_i) = 1 - \nu_\mathcal{A}$. As $(\mathbf{x}^*_i)_{i=1}^n$ is feasible we have $\prod_{i=1}^{n} (1 - \nu_i) \geq \prod_{i=1}^{n} (1 - \nu_i^*)$.
If $\nu_i = \nu_i^*$ for all $i\in[n]$, then convex combinations of $(\mathbf{x}_i)_{i=1}^n$ and $(\mathbf{x}_i^*)_{i=1}^n$ remain feasible, and by the arguments of \textit{Case 1}, $(\mathbf{x}_i)_{i=1}^n$ is strictly sub-optimal.
If there exists $i\in[n]$ such that $\nu_i \neq \nu_i^*$, then there exists $j$ such that $\nu_j < \nu_j^*$. In this case, we first construct intermediate occupancy measures by defining $\mathbf{x}_j' = \theta \mathbf{x}_j + (1-\theta) \mathbf{x}_j^*$ for some $\theta > 0$, and $\mathbf{x}_i' = \mathbf{x}_i$ for all $i \neq j$. 
By convexity of KL in the occupancy measure, we then have that $K_j(\mathbf{x}_j') \leq \max(K_j(\mathbf{x}_j),K_j(\mathbf{x}_j^*)) \leq \max_i \ K_i(\mathbf{x}_i)$ and as such, $\max_i \ K_i(\mathbf{x}_i') \leq \max_i \ K_i(\mathbf{x}_i)$. 
For $\theta > 0$, this construction produces point $\mathbf{x}'$ that satisfies reachability strictly. We can then appeal to \textit{Case 1} to produce point $\mathbf{x}''$ close to $\mathbf{x}'$ with a strictly lower value of the objective function.  Point $\mathbf{x}'$ can also be made arbitrarily close to $\mathbf{x}$. As such, we may make point $\mathbf{x}''$, satisfying $\max_i \ K_i(\mathbf{x}_i'') < \max_i \ K_i(\mathbf{x}_i') \leq  \max_i \ K_i(\mathbf{x}_i)$, arbitrarily close to $\mathbf{x}$, which implies that $\mathbf{x}$ cannot be a local minimum. $\square$

\paragraph{Proof of Theorem~\ref{the:LineSearchConvergenceBasic}.} 
Let $\nu_\Delta(K) = \textproc{ReachEvaluate}((\Pol{S}{i}, R_i^\mathcal{A})_{i=1}^n,\nu_\mathcal{A},K)$. 
We first note that
    $K_1 < K_2$ implies  $\textproc{Reach}(\pi^\mathcal{S}_i, R^\mathcal{A}_i, K_1) \leq \textproc{Reach}(\pi^\mathcal{S}_i, R^\mathcal{A}_i, K_2)$ for all $i\in[n]$, which implies $\nu_\Delta$ is increasing.
This inequality holds as the latter $\textproc{Reach}$ problem is a relaxation of the former for all $i\in[n]$. 

We now discuss the initial values of $\overline{K}$ and $\underline{K}$, which are $K_{\max}$ and $0$ respectively.
As $K_{\max}$ is a bound on the optimal value, $\nu_\Delta(K_{\max}) \geq 0$. 
If $\nu_\Delta(0)\geq 0$, it is optimal to use $\Pol{A}{i} = \Pol{S}{i}$ for all $i$. As such, we may assume $\nu_\Delta(0)$ < 0.

Consider the final values $\overline{K}_t$ and $\underline{K}_t$.
At each bisection iteration we compute $K = (\underline{K} + \overline{K})/2$. 
If $\nu_\Delta(K) < 0$, we then set $\underline{K} = K$, or if $\nu_\Delta(K) \geq 0$, we set $\overline{K} = K$.
By the $\textproc{Bisection}$ definition, and the fact that $\nu_\Delta(0) < 0$, we have $\nu_\Delta(\underline{K}_t) < 0$.
Using monotonicity of $\nu_\Delta$, we then conclude that there do not exist feasible policies that satisfy $\max_i \ KL(\mathbf{x}_i, \pi^\mathcal{S}_i) < \underline{K}_t$, and so, $K^* \geq \underline{K}_t$. 
Again, by the $\textproc{Bisection}$ definition and the fact that $\nu_\Delta(K_{\max}) \geq 0$, we have $\nu_\Delta(\overline{K}_t) \geq 0$ and we may observe that $K^* \leq \overline{K}_t$.
At termination, we have $\overline{K}_t \leq \underline{K}_t + \varepsilon$, as $\varepsilon$ is the tolerance. We finally conclude 
$K^* \leq \overline{K}_t \leq \underline{K}_t + \varepsilon \leq K^* + \varepsilon$. $\square$

\begin{credits}

\subsubsection{\ackname} This work was supported partially by the Army Research Laboratory (ARL), under grant number W911NF-23-1-0317, partially by the Defence
Advanced Research Projects Agency (DARPA), under grant number HR001123S0001,
and partially by the Oﬃce of Naval Research (ONR), under grant number
N000142412432.
\subsubsection{\discintname}
The authors have no competing interests.
\end{credits}
%
%
%
\bibliographystyle{splncs04}
\bibliography{MultiAgentDeception}
\end{document}